\documentclass[a4paper,10pt]{article}

\title{Handel's fixed point theorem revisited}
\author{Juliana Xavier}
\date{}

\usepackage{psfrag}
\usepackage{fancyhdr}
\usepackage{mathrsfs}
\usepackage{verbatim}
\usepackage[ansinew]{inputenc}
\usepackage{amsmath,amsthm,amssymb,amscd}
\usepackage[all]{xy}
\usepackage{xspace}
\usepackage[dvips]{color}
\usepackage{epsfig}
\usepackage{makeidx}
\usepackage{pb-diagram}
\usepackage{amsfonts}
\usepackage{graphicx}
\usepackage{subfigure}
\usepackage[absolute]{textpos}

\newcommand{\Z}{\mathbb{Z}}

\newcommand{\R}{\mathbb{R}}
\newcommand{\C}{\mathbb{C}}
\newcommand{\D}{\mathbb{D}}
\newcommand{\om}{\omega}

\DeclareMathOperator{\fix}{Fix}
\DeclareMathOperator{\per}{Per}
\DeclareMathOperator{\homeo}{Homeo}
\DeclareMathOperator{\inte}{Int}

\newtheorem{teo}{Theorem}[section]
\newtheorem{cor}[teo]{Corollary}
\newtheorem{lema}[teo]{Lemma}
\newtheorem{prop}[teo]{Proposition}

\theoremstyle{definition}

\newtheorem{obs}[teo]{Remark}
\newtheorem{ex}{Example}
\theoremstyle{remark}
\makeindex

\begin{document}

\maketitle

\begin{abstract} Michael Handel proved in \cite{handel} the existence
of a fixed point for an orientation preserving homeomorphism of the
open unit disk that can be extended to the closed disk, provided that
it has points whose orbits form an {\it oriented cycle of links at
infinity}.  Later, Patrice Le Calvez gave a different proof of this
theorem based only on Brouwer theory and plane topology arguments
\cite{patrice}.  These methods permitted to improve the result by
proving the existence of a simple closed curve of index 1.  We give a
new, simpler proof of this improved version of the theorem and
generalize it to non-oriented cycles of links at infinity.
\end{abstract}

\section{Introduction}\label{intro}

Handel's fixed point theorem \cite{handel} has been of great
importance for the study of surface homeomorphisms.  It guarantees
the existence of a fixed point for an orientation preserving
homeomorphism $f$ of the unit disk $\D = \{z\in \C : |z| < 1\}$
provided that it can be extended to the boundary $S^1 = \{z\in \C :
|z|=1\}$ and that it has points whose orbits form an oriented cycle
of links at infinity.  More precisely, there exist $n$ points $z_i
\in \D$ such that

$$\lim _{k \to -\infty} f^k(z_i) = \alpha _i \in S^1 ,  \ \lim _{k
\to +\infty} f^k(z_i) = \om _i \in S^1 ,$$ 

\noindent $i=1, \ldots, n$, where the $2n$ points $\{\alpha _i\}$, $\{\om _i\}$ are different points in $S^1$
and satisfy the following order property:

(*) $\alpha_{i+1}$ is the only one among these points that lies in the  open interval in the oriented circle $S^1$
from $\om_{i-1}$ to $\om_i$ .

\noindent (Although this is not Handel's original statement, it is an equivalent one as already pointed out in
\cite{patrice}).

Le Calvez gave an alternative proof of this theorem \cite{patrice},
relying only in Brouwer theory and plane topology, which allowed him
to obtain a sharper result. Namely, he weakened the extension
hypothesis by demanding  the homeomorphism to be extended just to $\D
\cup (\cup _{i \in \Z/n\Z} \{\alpha _i , \om _i\})$ and he strengthed
the conclusion by proving the existence of a simple closed curve of
index 1.

We give a new, simpler proof of this improved version of the theorem and we
generalize it to non-oriented cycles of links at infinity; that is, we relax 
the order property (*) as follows. \\

Let $P \subset \D$ be a compact convex $n$-gon.  Let $\{v_i : i\in \Z/n\Z\}$ be its set of vertices and for each 
$i\in \Z/n\Z$, let  $e_i$ be the edge joining
$v_i$ and $v_{i+1}$. We suppose that each $e_i$ is endowed with an
orientation, so that we can tell whether $P$ is to the right or to the left of $e_i$ .  We say that the orientations of
$e_i$ and $e_j$
{\it coincide} if $P$ is to the right (or to the left) of both $e_i$ and $e_j$, $i, j \in \Z/n\Z$. \\
We define the {\it index}\index{index!of a polygon} of $P$ by  

$$i(P) = 1 - \frac{1}{2} \sum _{i \in \Z/n\Z} \delta _i,$$

\noindent where $\delta _i = 0$ if the orientations of $e_{i-1}$ and $e_i$
coincide, and $\delta _i = 1$ otherwise.

We will note $\alpha _i$ and $\om _i$ the first, and respectively the last,
point  where the straight line $\Delta _i$ containing $e_i$ and inheriting its
orientation intersects $\partial \D$.

\begin{figure}[h]   
\begin{center}

   \subfigure[Handel's index 1 polygon]{\includegraphics[scale=0.5]{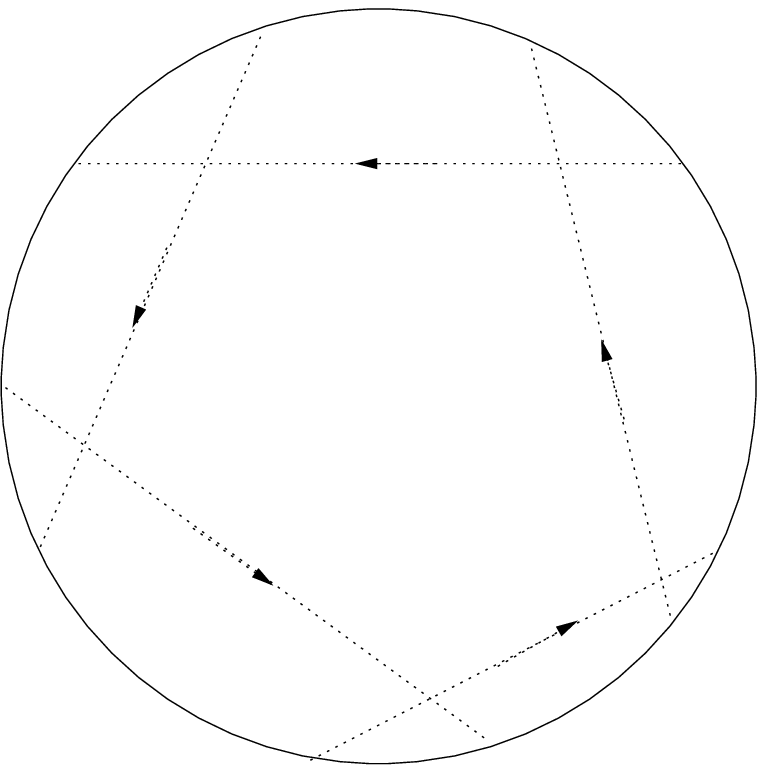}}\hspace{.25in}
    \subfigure[ Index -1 polygon] {\includegraphics[scale=0.5]{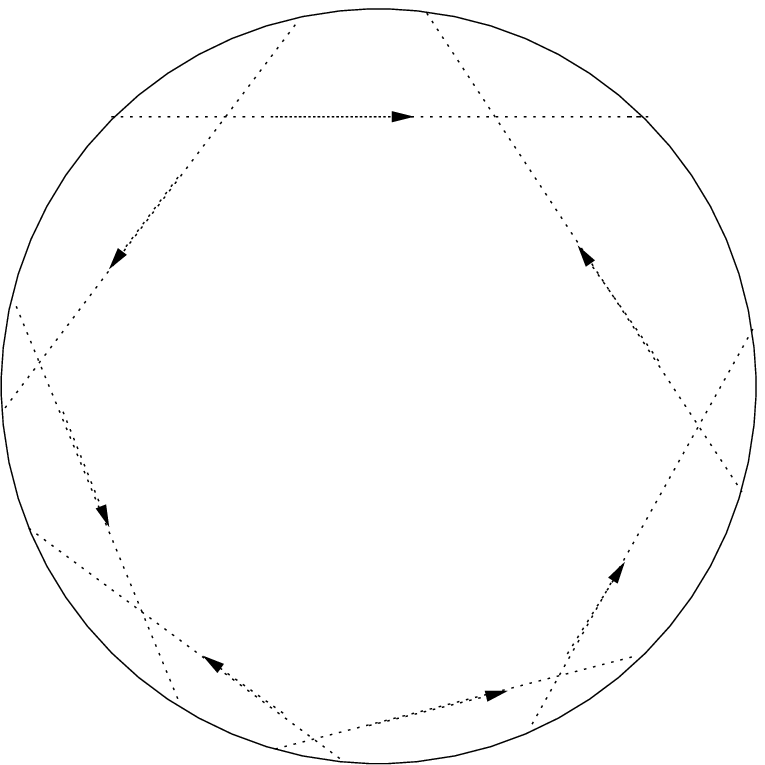}}\hspace{.25in}
    \subfigure[$\om _i = \alpha _{i+2} \ \forall i$ ] {\includegraphics[scale=0.5]{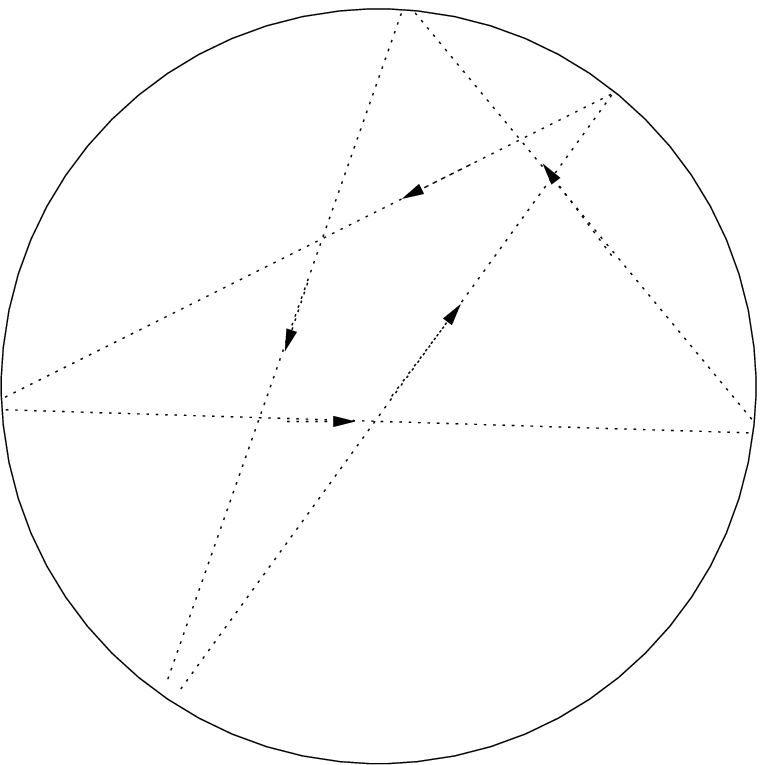}} \\
\caption {Polygons of different indices.}
\end{center}
 \end{figure}

\noindent We say that a homeomorphism $f:\D \to \D$ {\it realizes} $P$ if
there exists a family $(z_i)_{i\in \Z/n\Z}$ of points in $\D$ such
that for all $i\in \Z/n\Z$,
$$\lim _{k \to -\infty} f^k(z_i) = \alpha _i ,  \ \lim _{k \to
+\infty} f^k(z_i) = \om _i .$$\\

We will prove

\begin{teo}\label{main} Let $f : \D \to \D$ be an
orientation preserving homeomorphism  which realizes a compact convex
polygon $P\subset \D$ where the points $\alpha_i, \om_i, i\in \Z/n\Z$ are all different.  Suppose that $f$ can be extended to a homeomorphism of $\D
\cup (\cup _{i \in \Z/n\Z} \{\alpha _i , \om _i\}).$\\
If $i(P) \neq 0$, then $f$ has a fixed point.  Furthermore, if
$i(P)=1$, then there exists a simple closed curve $C\subset \D$ of
index  1 . 
 
\end{teo}

The two polygons appearing in Figure 1 (a) and (b) satisfy the hypothesis of this theorem.  However,  the polygon
illustrated in (c) does not, as there are coincidences among the points 
$\{\alpha_i\}, \{\om_i\}$, $i\in \Z/n\Z$.

I am endebted to Patrice Le Calvez.  Not only he suggested me
to study possible generalizations of Handel's theorem, but he guided
my research through a great number of discussions.

\section{Preliminaries}\label{bricks}

\subsection{Brick decompositions}

A {\it brick decomposition}\index{brick! decomposition} $\cal D$ of an orientable surface $M$ is
a $1$- dimensional singular submanifold $\Sigma (\cal D)$ (the {\it
skeleton} of the decomposition), with the property that the set of singularities  $V$ is discrete and such that every 
$\sigma \in V$ has a neighborhood $U$
for which $U\cap (\Sigma ({\cal D})\backslash V)$ has exactly
three connected components. We have illustrated two brick decompositions
in Figure 4. The {\it bricks}\index{brick} are the closure of the
connected components of $M \backslash \Sigma (\cal D)$ and the {\it
edges} are the closure of the connected components of $\Sigma({\cal D})\backslash V$. We will write $E$ for the set of edges, $B$ for
the set of bricks and finally $\cal D$ = $(V,E,B)$ for a brick
decomposition.

\begin{figure}[h]\label{brickdec}
\begin{center}

   \subfigure[$M = \R ^2$]{\includegraphics[scale=0.4]{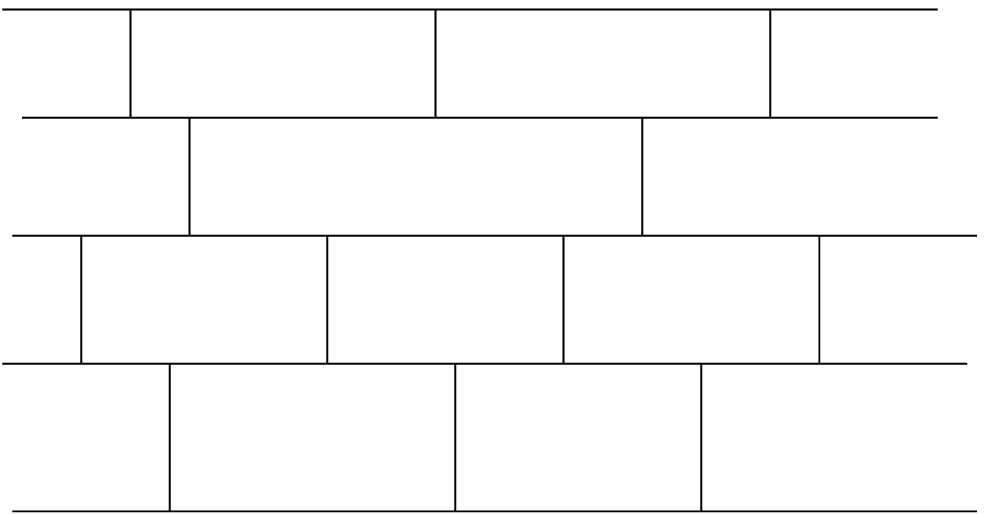}}\hspace{.25in}
    \subfigure[$M = \R ^2 \backslash \{0\}$]{\includegraphics[scale=0.5]{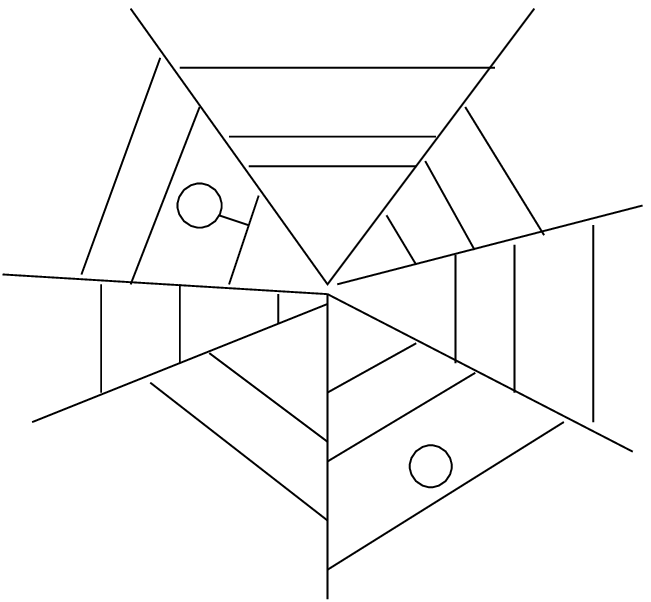}} \\
\caption{Brick decompositions}
\end{center}
 \end{figure}

Let $\cal D$ = $(V,E,B)$  be a brick decomposition of $M$.  We say
that $X\subset B$ is {\it connected} if given two bricks $b$, $b' \in X$,
there exists a sequence $(b_i)_{0\leq i \leq n}$, where $b_0 = b$,
$b_n = b'$ and such that $b_i$ and $b_{i+1}$ have non-empty intersection, $i\in \{0, \ldots, n-1\}$.  Whenever
two bricks $b$ and $b'$ have no empty intersection, we say that they are {\it adjacent}\index{brick! adjacent}.  Moreover, we say that a brick
$b$ is {\it adjacent to a subset} $X\subset B$ if $b\notin X$, but $b$ is adjacent to one of the bricks in $X$.  We say that
$X\subset B$ is adjacent to $X'\subset B$ if $X$ and $X'$ have no common bricks but there exists $b\in X$ and $b'\in X'$
which are adjacent.

From now on we will identify a subset $X$ of $B$ with the closed subset of $M$ formed by
the union of the bricks in $X$.  By making so, there may be ambiguities (for instance, two adjacent subsets of $B$ have empty 
intersection in $B$ and nonempty intersection in $M$), but we will point it out when this happens. We remark that $\partial X$ is a
one-dimensional topological manifold and that the connectedness of
$X\subset B$ is equivalent to the connectedness of $X\subset M$ and
to the connectedness of $\inte (X) \subset M$ as well. We say that the decomposition $\cal D '$ is a 
{\it subdecomposition}\index{brick! subdecomposition}
of $\cal D$ if $\Sigma (\cal D ')$ $\subset \Sigma (\cal D)$.

If $f:M\to M$ is a homeomorphism, we define the application $\varphi
: {\cal P} (B) \to {\cal P} (B)$ as follows:
$$\varphi (X) = \{b\in B : f(X)\cap b
\neq \emptyset\}.$$\\

We remark that $\varphi (X)$ is connected whenever $X$ is.

We define analogously an application $\varphi _-: {\cal P} (B) \to {\cal
P} (B)$:

$$\varphi _- (X) =  \{b\in B :
f^{-1}(X)\cap b \neq \emptyset\}.$$\\

\begin{figure}[h]
\begin{center}
\psfrag{a}{$b$}\psfrag{b}{$f(b)$}\psfrag{c}{$\varphi (\{b\})$}

   {\includegraphics[scale=0.5]{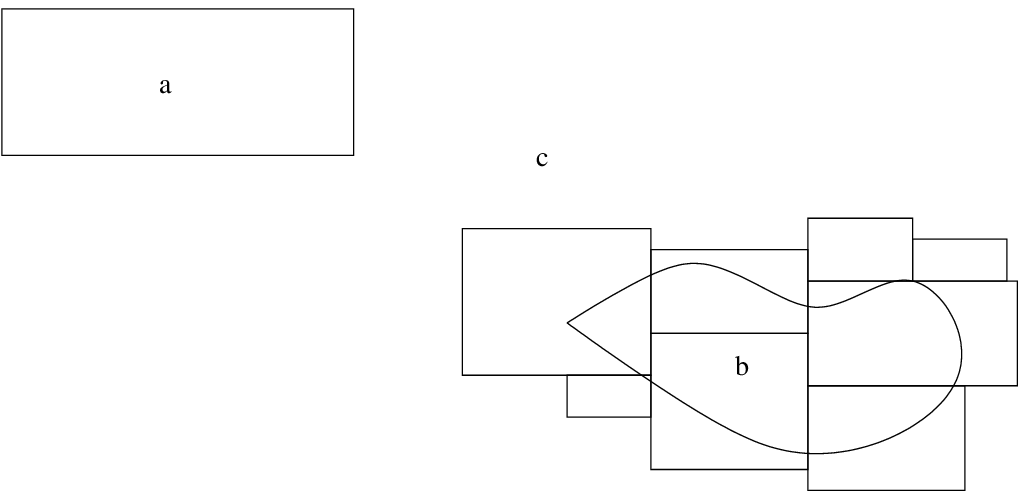}}
  \\

\end{center}
 \end{figure}

We define the {\it future}\index{brick! future} $[b]_\geq$ and the {\it past}\index{brick! past} $[b]_\leq$
of a brick $b$ as follows:

$$[b]_\geq = \bigcup _{k\geq 0} \varphi ^k (\{b\}),\  [b]_\leq =
\bigcup _{k\geq 0} \varphi _- ^k (\{b\}).$$

We also define the {\it strict future} $[b]_>$ and the {\it strict
past} $[b]_<$ of a brick $b$ :
$$[b]_> = \bigcup _{k>0} \varphi ^k (\{b\}),\  [b]_< = \bigcup
_{k> 0} \varphi _- ^k (\{b\}).$$

We say that a set $X\subset B$ is an {\it attractor}\index{attractor} if it verifies
$\varphi (X) \subset X$; this is equivalent in $M$ to the inclusion
$f(X)\subset \inte(X)$.  A {\it repeller}\index{repeller} is any set which verifies
$\varphi _- (X) \subset X$.  In this way, the future of any brick is an attractor,
and the past of any brick is a repeller. We observe that $X\subset B$
is a repeller if and only if $B\backslash X$ is an attractor.\\

\begin{obs}\label{br} The following properties can be deduced from the fact that $X\subset B$ is an attractor if and only if 
$f(X)\subset \inte (X)$:

\begin{enumerate}
 \item\label{br1} If $X\subset B$ is an attractor and $b\in X$, then $[b]_\geq \subset X$ ; 
if $X\subset B$ is a repeller and $b\in X$, then $[b]_\leq \subset X$,
\item\label{br2} if  $X\subset B$ is an attractor and $b\notin X$, then $[b]_\leq \cap X = \emptyset$ ;
if  $X\subset B$ is a repeller and $b\notin X$, then $[b]_\geq \cap X = \emptyset$,
\item\label{br3} if $b\in B$ is adjacent to the attractor $X\subset B$, then $[b]_> \cap X\neq \emptyset$;  
if $b\in B$ is adjacent to the repeller $X\subset B$, then $[b]_< \cap X\neq \emptyset$; 
\item\label{br4} two attractors are disjoint as subsets of $B$ if and only if they are disjoint as subsets of $M$;
in other words, two disjoint (in $B$) attractors  cannot be adjacent; respectively two disjoint (in $B$) repellers 
cannot be adjacent;
\end{enumerate}
 
\end{obs}

The following conditions are equivalent:
$$b\in [b]_>,\ [b]_> = [b]_\geq,\ b\in [b]_<,\ [b]_< = [b]_\leq,\ [b]_<\cap
[b]_\geq \neq \emptyset,\ [b]_\leq\cap [b]_> \neq \emptyset.$$

The existence of a brick $b\in B$ for which any of these conditions is 
satisfied is equivalent to the existence of a {\it closed chain of
bricks }, i.e a family  $(b_i)_{i\in \Z/r\Z}$ of bricks such that for
all $i\in \Z/r\Z$, $\cup _{k \geq 1} f^k(b_{i}) \cap b_{i+1}\neq
\emptyset$.  

In general, a {\it chain}\index{chain} for $f\in \homeo (M)$ is a family
$(X_i)_{0\leq i \leq r}$ of subsets of $M$ such that for all $0\leq i
\leq r-1$ ,
$\cup _{k\geq 1} f^k(X_i) \cap X_{i+1}\neq \emptyset$. We say that
the chain is
closed if $X_{r}= X_{0}$.

We  say that a subset $X\subset M$  is {\it
free}\index{free! set} if $f(X)\cap X= \emptyset$.

We  say that a brick decomposition $\cal D$ = $(V,E,B)$  is {\it
free}\index{free!brick decomposition} if every $b\in B$ is a free subset of $M$. If $f$ is
fixed point free it is always possible, taking sufficiently small
bricks, to construct a free brick decomposition.\\

We recall the definition of {\it maximal free decomposition}\index{maximal! free decomposition}, which
was introduced by Sauzet in his doctoral thesis \cite{sauzet}.  Let
$f$ be a fixed point free homeomorphism of a surface $M$. We say that
$\cal D$ is a maximal free decomposition if $\cal D$ is free and any
strict subdecomposition is no longer free.  Applying Zorn's lemma, it
is always possible to construct a maximal free 
subdecomposition of a given 
brick decomposition $\cal D$.

\subsection{Brouwer Theory background.}

We say that $\Gamma : [0,1]\to \overline \D$ is an {\it arc}\index{arc}, if it is continuous and injective.  We say that an arc $\Gamma$
joins $x\in \overline \D$ to $y\in \overline \D$, if $\Gamma (0)= x $ and $\Gamma (1)= y $.  We say that an arc $\Gamma$
joins $X\subset \overline \D$ to $Y\subset \overline \D$, if $\Gamma$ joins $x\in X$ to $y\in Y$.\\

Fix $f\in \homeo ^+ (\D)$.  An arc $\gamma$ joining $z\notin \fix (f)$ to $f(z)$ such that 
$f(\gamma)\cap \gamma = \{z, f(z)\}$ if $f^2(z) = z$ and $f(\gamma)\cap \gamma = \{ f(z)\}$ otherwise, is called a 
{\it translation arc}\index{arc! translation}.

\begin{prop}{\bf (Brouwer's translation lemma \cite{brouwer}, \cite{brown}, \cite{fathi} or \cite{guillou})}  If any
of the two following hypothesis is satisfied, then there exists a  simple closed curve 
of index 1:

\begin{enumerate}
 \item  there exists  a translation arc $\gamma$ joining  $z\in \fix (f^2)\backslash \fix (f)$ to $f(z)$;
\item  there exists  a translation arc $\gamma$ joining $z\notin \fix (f^2)$ to 
$f(z)$ and an integer $k\geq 2$ such that  $f^k(\gamma) \cap \gamma \neq \emptyset$.
\end{enumerate}

\end{prop}

If $z\notin \fix (f)$, there exists a translation arc containing $z$; this is easy to prove once one has that the connected 
components of the complement of $\fix (f)$ are invariant.  For a proof of this last fact, see \cite{brownkister} for a general proof in any dimension, or \cite{duke}
for an easy proof in dimension 2.

We deduce:

\begin{cor} If $\per (f) \backslash \fix (f) \neq \emptyset$, then  there exists a  simple closed curve 
of index 1. 
 
\end{cor}

\begin{prop} {\bf (Franks' lemma \cite{pbf})} If there
exists a closed chain of free, open and pairwise disjoint disks for $f$, then there exists a simple closed curve of index 1.
 
\end{prop}

Following Le Calvez \cite{patrice}, we will say that $f$ is {\it recurrent}\index{recurrent homeomorphism} if there
exists a closed chain of free, open and pairwise disjoint disks for $f$.\\

The following proposition is a refinement of Franks' lemma due to Guillou and  Le Roux (see \cite{leroux}, page 39).

\begin{prop}\label{guile} Suppose there exists a closed chain
$(X_i)_{i\in\Z/r\Z}$ for $f$ of free subsets  whose
interiors are pairwise disjoint and which verify the following
property: given any two points $z,z'\in X_i$ there exists an
arc $\gamma$ joining $z$ and $z'$ such that $\gamma\backslash\{z,z'\}\subset
\inte (X_i)$. Then, $f$ is recurrent.
 
\end{prop}

We deduce:
\begin{prop}\label{franksfino} Let $\cal D$ = $(V,E,B)$  be a free
brick decomposition of $\D \backslash \fix (f)$.  If there exists
$b\in B$ such that $b\in [b]_>$, then $f$ is recurrent.
 
\end{prop}

\subsection{Previous results.}

Fix  $f\in \homeo ^+ (\D)$, different from the identity map and  {\it non-recurrent}. We will make use of the following two propositions from
\cite{patrice} (both of them depend on the non-recurrent character of $f$).  The first one (Proposition 2.2 in \cite{patrice})
 is a refinement of a result already appearing in \cite{sauzet}; the second one is Proposition 3.1 in \cite{patrice}.

\begin{prop}[\cite{sauzet},\cite{patrice}]\label{futcon}  Let $\cal D$ = $(V,E,B)$  be a
free maximal brick decomposition of $\D\backslash \fix (f)$.  Then,
 the sets $[b]_\geq$, $[b]_>$, $[b]_\leq$ and $[b]_<$ are connected.  In particular every connected component of an attractor
is an attractor, and  every connected component of a repeller is a repeller.

\end{prop}

\begin{prop}\cite{patrice}\label{gamak} If $f$  satisfies the
hypothesis of Theorem \ref{main*}, then for all $i\in \Z/n\Z$ we can find a sequence of arcs
$(\gamma _i^k)_{k\in\Z}$ such that:
\begin{enumerate}\item[$\bullet$] each $\gamma _i^k$ is a translation
arc from $f^k(z_i)$ to $f^{k+1}(z_i)$,
\item[$\bullet$] $f(\gamma _i^k)\cap \gamma _i^{k'} = \emptyset$ if
$k'<k$,
\item[$\bullet$] the sequence $(\gamma _i^k)_{k\leq 0}$ converges to
$\{\alpha _i\}$ in the Hausdorff topology,
\item[$\bullet$] the sequence $(\gamma _i^k)_{k\geq 0}$ converges to
$\{\om _i\}$ in the Hausdorff topology.
\end{enumerate}

\end{prop}

This result is  a consequence of
Brouwer's translation lemma and the hypothesis on the orbits of the
points $(z_i)_{i\in\Z/n\Z}$. In particular,   the
extension hypothesis of Theorem \ref{main*} is used. It
allows us to construct a particular
brick decomposition suitable for our purposes: \\

\begin{lema}\label{ends} For every $i\in \Z/n\Z$, take $U_i^-$ a neighbourhood of $\alpha _i$ in $\overline \D$ and 
$U_i^+$ a neighbourhood of $\om _i$ in $\overline \D$ such that $U_i^-\cap U_i^+ = \emptyset$.
There exists two families $(b_i'^l)_{i\in\Z/n\Z,l\geq 1}$ and
$(b_i'^l)_{i\in\Z/n\Z,l\leq -1}$ of closed disks in $\D$, and a family of integers $(l_i)_{i\in \Z/n\Z}$ such that:
\begin{enumerate}\item each $b_i'^l$ is free and contained in $U_i^-$
($l\leq-1$) or in $U_i^+$ ($l\geq 1$),
\item $\inte (b_i'^l) \cap \inte (b_i'^{l'}) = \emptyset$, if $l \neq l'$
,
\item for every $k>1$ the sets $(b_i'^l)_{1\leq l \leq k}$ and
$(b_i'^l)_{-k \leq l \leq -1 }$ are connected, 

\item for all $i\in\Z/n\Z$, $\partial \cup_{l\in \Z\backslash \{0\}} b_i'^l$ is a one dimensional submanifold,
\item if $x\in \D$, then $x$ belongs to at most two different disks in the family $(b_i'^l)_{l\in \Z\backslash \{0\}}$,
$i\in\Z/n\Z$,
\item\label{pto1} for all  $i\in \Z/n\Z$  $f^{l_i + l} (z_i) \in \inte(b_i'^{l+1}) $ for all $l\geq 0$,
and  $f^{-l_i-l} (z_i) \in
\inte(b_i'^{-l-1}) $ for all $l\geq 0$,
\item $f^k(z_j) \in b_i'^l$ if and only if $j=i$ and $k = l_i+l-1$,
\item the sequence $(b_i'^l)_{l \geq 1}$ converges to  $\{\om _i\}$ in the Hausdorff topology
and the sequence $(b_i'^l)_{l \leq -1}$ converges to $\{\alpha _i\}$ in the Hausdorff topology.

\end{enumerate}

\end{lema}

The idea is to construct  trees $T_i^-\subset U_i^-, T_i^+ \subset U_i^+$, $i\in \Z/n\Z$ by deleting the loops of the curves 
$\prod _{k\geq -1} \gamma _i^k\cap U_i^-$ and 
 $\prod _{k\leq 1} \gamma _i^k\cap U_i^+$ respectively, and then
thickening these trees to obtain the families $(b_i'^l)_{i\in\Z/n\Z,l\geq 1}$
and $(b_i'^l)_{i\in\Z/n\Z,l\leq -1}$. We have illustrated these families in Figure 5. 
Given the centrality of this lemma to this paper, a proof will be included
at the end of this section.  We remark, however that these results are contained in \cite{patrice}.

\begin{figure}[h]\label{fliabil}
\begin{center}
\psfrag{a}{$b_3'^{-l}$}\psfrag{b}{$b_1'^l$}\psfrag{c}{$b_0'^l$}\psfrag{d}{$b_2'^{-l}$}\psfrag{e}{$b_1'^{-l}$}\psfrag{f}{$b_3'^l$}
\psfrag{g}{$b_2'^{l}$}
\psfrag{h}{$b_0'^{-l}$}\psfrag{i}{$\alpha_3$}\psfrag{j}{$\om_1$}\psfrag{k}{$\om_0$}\psfrag{l}{$\alpha_2$}\psfrag{m}{$\alpha_1$}
\psfrag{n}{$\om_3$}\psfrag{o}{$\om_2$}\psfrag{p}{$\alpha_0$}
\includegraphics[scale=0.6]{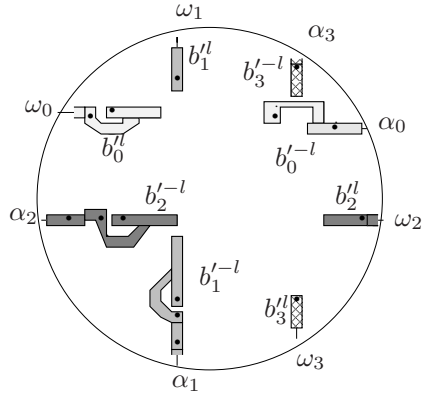}
\caption{The families $b_i'^l$}
\end{center}
\end{figure}

\begin{obs}\label{puntas}  The fact that the sequence $(b_i'^l)_{l\geq 1}$ converges in the 
Hausdorff topology to $\om _i$, implies that we can find an 
arc  $\Gamma_i ^+ : [0,1]\to \inte (\cup _{l\geq 0 } 
b_i'^{l}) \cup \{\om _i\}$ such that  $\Gamma_i ^+(1) = \om _i$, $i\in \Z/n\Z$.  Similarly, we can find an arc
$\Gamma_i ^- : [0,1]\to \inte (\cup _{l\geq 0 } b_i'^{-l}) \cup \{\alpha _i\}$ such that $\Gamma_i ^-(1) = \alpha _i$, 
$i\in\Z/n\Z$.
\end{obs}

\begin{obs}\label{ptsdf} If the points $\alpha _i, \om_i$, $i\in \Z/n\Z$, are all different, the bricks $b_i'^l$, $i\in \Z/n\Z$, 
$l\in \Z\backslash\{0\}$ can be constructed as to have pairwise disjoint interiors.
 
\end{obs}

\begin{cor}\label{bon} If the points $\alpha _i, \om_i$, $i\in \Z/n\Z$, are all different, there exists a free brick 
decomposition $(V,E,B)$ of $\D\backslash \fix(f)$ such that for all $i\in\Z/n\Z$ and all $l\in\Z\backslash\{0\}$, there exists
$b_i^l\in B$ such that $b_i'^l\subset b_i^l$.
 
\end{cor}

We will make use of proposition \ref{futcon} in the next section.  Propositions \ref{gamak} and 
\ref{ends} will not be used until section \ref{proof}.  

We finish this section with a proof of Proposition \ref{gamak}, due to Le Calvez \cite{patrice}.  Fix $f\in \homeo^+(\D)$ 
satisfying the hypothesis of Theorem \ref{main*}.  Let
$O_i$ be the orbit of $z_i$, and $z_i^k = f^k(z_i)$, $k\in\Z$.  We will need the three following lemmas, where
we will omit the index $i$ for simplicity.

\begin{lema} There exists a sequence of pairwise disjoint arcs $(\gamma '^{k})_{k\in\Z}$ such that:

\begin{itemize}
 \item $z^k\in \gamma'^k$;
\item $\gamma'^k \cap \fix(f) = \emptyset$:
\item $f(\gamma'^k)\cap \gamma'^k \neq \emptyset$;
\item the sequence $(\gamma '^{k})_{k\leq 0}$ converges to $\{\alpha\}$;
\item the sequence $(\gamma '^{k})_{k\geq 0}$ converges to $\{\om\}$.
\end{itemize}

\end{lema}

\begin{proof} We can construct a homeomorphism $h: \D\to (-1,1)^2$ such that:

\begin{itemize}
 \item $\lim_l p_1(x_l) =-1 \Leftrightarrow \lim _l h^{-1}(x_l) =\alpha$; 
\item $\lim_l p_1(x_l) = 1 \Leftrightarrow \lim _l h^{-1}(x_l) =\om$, where $p_1, p_2$ 
are the projections to the horizontal and vertical coordinates;
\item $p_1$ is injective on $h(O)$, where $O$ is the orbit of $z$;
\item the sequence $p_2(h(z^k))_{k\in\Z}$ is increasing.
\end{itemize}

Indeed, it is easy to construct a homeomorphism $h': \D\to (-1,1)^2$ satisfying the two first items. As
$$\lim_{k\to -\infty} z^k = \{\alpha\}, \lim_{k\to \infty} z^k = \{\om\}, $$ \noindent for any $k$ there is only
a finite number of points in $h'(O)$ in the same vertical line that $h'(z^k)$, and so we can perturb $h'$ in a homeomorphism
$h''$ satisfying the three first items.  Once one has injectivity of $p_1$ on $h''(O)$, we can compose $h''$ with a 
homeomorphism fixing each vertical line to have $p_2$ increasing on $h''(O)$.

For simplicity, we will no longer write $h$;  we will suppose that $p_1$ and $p_2$ 
are defined in $\D$.
Let $I_k$ be the open interval of $(-1,1)$ delimited by $p_1(z^k)$ and $p_1(z^{k+1})$, and 
$U^k = I_k \times (p_2(z^k), p_2 (z^{k+1}))$.  Then, $(\overline{U^k})_{k\leq 0}$
and $(\overline{U^k})_{k\leq 0}$ are sequences of closed disks in $\D$ converging
respectively to $\{\alpha\}$ and $\{\om\}$.

We can pick a point $z'^0\in U^0$ such that $f(z'^0)$ does not belong to the same
vertical line that any of the points in $O$.  We can also pick a point 
$z'^1\in U^1$ such that $f(z'^1)$ does not belong to the same
vertical line that any of the points in $O$, and such that  $f(z'^1)$ does not belong
either to the same vertical line as $z'^0$, or the image of this vertical line or its preimage.
We can define inductively a sequence $(z'^k)_{k\in\Z}$ such that:

\begin{itemize}
 \item $z'^k \in U^k$;
\item $p_1(f(z'^k))\neq p_1(z^{k'})$ for all $k,k'\in \Z$;
\item $p_1(z'^{k})\neq p_1(z'^{k'})$ if $k\neq k'$;
\item  $p_1(f(z'^k))\neq p_1(z'^{k'})$ if $k\neq k'$.
\end{itemize}

So, we can modify $h$ (by composition with a homeomorphism fixing each vertical line, as we did before) so as to have 
$p_2(z^k)<p_2(f(z'^k))<p_2(z^{k+1})$. 

{\it The arguments that follows depends on the extension hypothesis of Theorem \ref{main*}}.
As $f$ extends to a homeomorphism of $\D\cup \{\alpha, \om\}$, and the sequences
 $(z'^k)_{k\leq 0}$ and  $(z'^k)_{k\geq 0}$ converge respectively to $\{\alpha\}$, and $\{\om\}$,
we obtain that the sequences   $(f(z'^k))_{k\leq 0}$ and  $(f(z'^k))_{k\geq 0}$ also converge  to $\{\alpha\}$, and $\{\om\}$
respectively.   It follows
that one can construct a sequence $(I'^k)_{k\in\Z}$ of open intervals of $(-1,1)$ such that:

\begin{itemize}
 \item $I^k\subset I'^k$;
\item $U'^k = I'^k \times (p_2(z^k), p_2(z^{k+1}))$ contains $z'^k$ and $f(z'^k)$;
\item the sequences of closed disks  $(\overline{U'^k})_{k\leq 0}$
and $(\overline{U'^k})_{k\leq 0}$  converge respectively to $\{\alpha\}$ and $\{\om\}$.
\end{itemize}

We will construct our arcs $\gamma '^k$ to be contained in $U'^k\cup \{z^k\}$.  So, these
arcs will be pairwise disjoint and the sequences $(\gamma '^k)_{k\leq 0}$ and $(\gamma '^k)_{k\geq 0}$
will converge respectively to $\{\alpha\}$ and $\{\om\}$.

If there is only a finite number of fixed points in $U'^k$, we can suppose that $z'^k$ is not fixed and take
an arc $\gamma '^k\subset U'^k\cup \{z^k\}$ disjoint from $\fix(f)$, with an endpoint in $z^k$,
and containing both $z'^k$ and $f(z'^k)$.

If there are infinitely many fixed points in $U'^k$, we can construct three arcs contained in
 $U'^k\cup \{z^k\}$, each one of them with an endpoint in $z^k$ and the other one in a fixed point, such
that these arcs meet only in $z^k$.  We can also suppose that the only fixed point of these arcs
is their other endpoint.  If one of these arcs meets its image outside its fixed extremity, we can find a
subarc $\gamma '^k$ disjoint from the fixed point set and meeting its image as we want.  Otherwise, as $f$ is
orientation preserving, necessarily the union of two of these three segments must meet its image outside the 
fixed point set.  If we delete a neighbourhood of the fixed extremity for both ot these arcs, we obtain
our arc $\gamma '^k$.

\end{proof}

By thickening the arcs given by the preceeding lemma, and then taking the ``smallest'' disk which is no longer free, we obtain:

\begin{lema}\label{dprima} There exists a sequence of pairwise disjoint closed disks $(D '^{k})_{k\in\Z}$ such that:

\begin{itemize}
 \item $z^k\in \partial D'^k$;
\item $D'^k \cap \fix(f) = \emptyset$:
\item $f(D'^k)\cap D'^k \neq \emptyset$;
\item $f(\inte((D'^k)))\cap \inte(D'^k) = \emptyset$;
\item the sequence $(D '^{k})_{k\leq 0}$ converges to $\{\alpha\}$;
\item the sequence $(D '^{k})_{k\geq 0}$ converges to $\{\om\}$.
\end{itemize}

\end{lema}

This last lemma allows us to construct the desired translation arcs.

\begin{lema}  Suppose that $f$ is not recurrent.  Then, there exists a sequence of pairwise disjoint closed disks 
$(D ^{k})_{k\in\Z}$ such that:

\begin{itemize}
 \item $z^k\in \inte (D^k)$;
\item $D^k \cap \fix(f) = \emptyset$:
\item $f(D^k)\cap D^k \neq \emptyset$;
\item $f(D^k)\cap D^{k'} = f^2(D^k)\cap D^{k'} =  \emptyset$ if $k'<k$;
\item the sequence $(D^{k})_{k\leq 0}$ converges to $\{\alpha\}$;
\item the sequence $(D^{k})_{k\geq 0}$ converges to $\{\om\}$.
\end{itemize}
 
\end{lema}

\begin{proof} Let $(D '^{k})_{k\in\Z}$ be the sequence of pairwise disjoint closed disks given by Lemma \ref{dprima}.  If $\gamma$ is an
arc joining $z^k$ and a point $z'\in \partial D'^k$ which is contained in $\inte (D'^k)$ except for its endpoints, then $\gamma$ is free.
Indeed as $\inte (D'^k)$ is free, $f(\gamma)\cap \gamma \neq \emptyset$ implies either $z^k \in f(\gamma)\cap \gamma $ or
$z'^k \in f(\gamma)\cap \gamma$.  The first case is impossible because $z^{k-1}$, the preimage of $z^k$, is contained in $D'^{k-1}$ which
is disjoint from $D'^k$.  The second case implies (as $D'^k\cap \fix(f)\neq \emptyset$) that $z'^k = f(z^k)$, which is also imposible as  
$z^{k+1}$ is contained in $D'^{k+1}$ which is disjoint from $D'^k$. 

Take a point $x_k\in \partial D'^k\cap f^{-1}(\partial D'^k)$, and two arcs $\gamma _-^k$, $\gamma_+^k$ contained in $\inte (D'^k)$ except 
for its endpoints, the former joining $x^k$ and $z^k$, and the latter  joining $z^k$ and $f(x^k)$, and such that $\gamma _-^k\cap \gamma_+^k =
\{z_k\}$.  If $k'<k$, then the positive orbit of $\gamma _-^{k'}$ and $\gamma_+^{k'}$ meets $\gamma _-^k$ and $\gamma_+^k$.  As these arcs
are all free, and we are supposing that $f$ is not recurrent, we obtain that the positive orbit of  $\gamma _-^k$ and $\gamma_+^k$ never
meets  $\gamma _-^{k'}$ or $\gamma_+^{k'}$.  Besides, as

$$\lim_{k\to -\infty} \gamma_-^k\gamma_+^k = \{\alpha\}, \lim_{k\to -\infty} \gamma_-^k\gamma_+^k = \{\alpha\}, $$
\noindent we can find a closed disk $D^0$ neighbourhood of $\gamma_-^0\gamma_+^0$ such that:

\begin{itemize}
 \item $D^0\cap \fix(f) =\emptyset$;
\item $D^0\cap \gamma_-^k\gamma_+^k = f(D^0)\cap \gamma_-^k\gamma_+^k = f^2(D^0)\cap \gamma_-^k\gamma_+^k = \emptyset$, if $k<0$;
\item $D^0\cap \gamma_-^k\gamma_+^k = f^{-1}(D^0)\cap \gamma_-^k\gamma_+^k = f^{-2}(D^0)\cap \gamma_-^k\gamma_+^k = \emptyset$, if $k>0$.
\end{itemize}

We obtain:

\begin{itemize}
 \item $z^0\in \inte(D^0)$;
\item $f(D^0\cap D^0)\neq \emptyset$.
\end{itemize}

Now we can choose a closed disk $D^1$ neighbourhood of $\gamma_-^1\gamma_+^1$ such that:

\begin{itemize}
 \item $D^1\cap \fix(f) =\emptyset$;
\item $D^1\cap \gamma_-^k\gamma_+^k = f(D^1)\cap \gamma_-^k\gamma_+^k = f^2(D^1)\cap \gamma_-^k\gamma_+^k = \emptyset$, if $k<1$;
\item $D^1\cap D^0 = f(D^1)\cap D^0= f^2(D^1)\cap D^0 = \emptyset$;
\item $D^1\cap \gamma_-^k\gamma_+^k = f^{-1}(D^1)\cap \gamma_-^k\gamma_+^k = f^{-2}(D^0)\cap \gamma_-^k\gamma_+^k = \emptyset$, if $k>1$.
\end{itemize}

So, 

\begin{itemize}
 \item $z^1\in \inte(D^1)$;
\item $f(D^1\cap D^1)\neq \emptyset$.
\end{itemize}

We proceed inductively to construct our sequence $(D ^{k})_{k\in\Z}$.

\end{proof}

Now we are ready to prove Proposition \ref{gamak}:

\begin{proof} Suppose that $f$ is non-recurrent, and take a sequence of closed disks $(D ^{k})_{k\in\Z}$ as in the preceding lemma.  By taking
a smaller disk if necessary, we can suppose that the interior of each $D^k$ is free.  Take a point 
$x_k\in \partial D'^k\cap f^{-1}(\partial D'^k)$, and two arcs $\gamma _-^k$, $\gamma_+^k$ contained in $\inte (D'^k)$ except 
for one endpoint, the former joining $x^k$ and $z^k$, and the latter  joining $z^k$ and $f(x^k)$, and such that $\gamma _-^k\cap \gamma_+^k =
\{z_k\}$.  Then, $\gamma_-^k\gamma_+^k$ is a translation arc.  As $f$ is not recurrent, $\gamma ^k = \gamma_-^k\gamma_+^k$ is a translation arc 
as well.  Besides, $\gamma ^k$ joins $z^k$ and $z^{k+1}$.  The other required properties of $\gamma ^k$ are verified because:

\begin{itemize}
 \item $\gamma^k \subset D^k\cup f(D^k)$ and $f$ extends to a homeomorphism of $\D\cup \{\alpha, \om\}$;
\item $D^{k'}\cup f (D^{k'})$ is disjoint from $f(D^{k})\cup f^2 (D^{k})$ if $k'<k$.
\end{itemize}

\end{proof}

The following lemma allows us to suppose that $\gamma _i^k\cap O_i = \{z_i^k, z_i^{k+1}\}$ and $\gamma _i^k\cap O_{i'} = \emptyset$ if $i'\neq i$.

\begin{lema}  Let $\gamma$ be a translation arc for a point $z\notin \fix (f^2)$ and $\gamma ' \subset \gamma \backslash \{z, f(z)\}$ an arc.
There exists a neighbourhood $U$ of $\gamma '$ such that any arc joining $z$ and $f(z)$ contained in $\gamma \cup U$ is a translation arc.
 
\end{lema}

\begin{proof} Just note that $$f(\gamma ' \cap \gamma ') = f(\gamma ')\cap \gamma = \gamma '\cap f(\gamma) = \emptyset.$$
 
\end{proof}

\section{Repeller/Attractor configurations at infinity}\label{raconf}

\subsection{Cyclic order at infinity.}\label{co}

Let $(a_i)_{i\in \Z/n\Z}$ be a family of  non-empty, pairwise disjoint, closed, connected subsets of 
$\D$, such that $\overline a_i \cap \partial \D \neq \emptyset$ and $U= \D\backslash (\cup_{i\in \Z/n\Z} a_i)$ is
a connected open set.  As $U$ is connected, and its complementary set in $\C$
$$\{z\in \C: |z|\geq 1\}\cup \cup_{i\in \Z/n\Z} a_i$$\noindent is connected, $U$ is simply connected.   

With these hypotheses, there is a natural cyclic order on the sets $\{a_i\}$.  Indeed, $U$ is conformally isomorphic
to the unit disc via the Riemann map $\varphi : U \to \D$, and one can consider the Carath\'eodory's extension of $\varphi$,

$$\hat\varphi: \hat U \to \overline{\D},$$\noindent which is a homeomorphism between the prime ends completion $\hat U$
of $U$ and the closed unit disk $\overline {\D}$.  The set $\hat J_i$ of prime ends whose impression is contained in $a_i$ is open and connected.  It follows that
the images $J_i = \hat\varphi (\hat J_i)$ are pairwise disjoint open intervals in $S^1$, and are therefore cyclically ordered
following the positive orientation in the circle.

\subsection{Repeller/Attractor configurations. }

We fix $f\in \homeo ^+ (\D)$ together with a free maximal decomposition in bricks  $\cal
D$$=(V,E,B)$ of $\D\backslash \fix (f)$ .\\

Let  $(R_i)_{i\in \Z/n\Z}$ and  $(A_i)_{i\in \Z/n\Z}$ be two families of connected, pairwise disjoint subsets
 of $B$ such that  :

\begin{enumerate}
 
\item For all $i\in \Z/n\Z$:

\begin{enumerate} \item  $R_i$ is a repeller and $A_i$ is an attractor;

\item there exists non-empty, closed, connected subsets of $\D$, $r_i\subset \inte(R_i)$, $a_i\subset
\inte(A_i)$
such that $\overline{r_i} \cap \partial \D\neq \emptyset$ and $\overline{a_i} \cap \partial \D\neq \emptyset$ ,
\end{enumerate}

\item \label{orden} $\D\backslash (\cup_{i\in \Z/n\Z}(a_i\cup r_i))$ is a connected open set.
\end{enumerate}

We say that the pair $((R_i)_{i\in \Z/n\Z} , (A_i)_{i\in \Z/n\Z})$ is a {\it Repeller/Attractor configuration of order $n$}
\index{Repeller/Attractor configuration}.\\ 
We will note 
$${\cal{E}} = \{R_i, A_i: i\in \Z/n\Z \}.$$

Property \ref{orden} in the previous definition allows us to  give a cyclic order to the sets $r_i, a_i, i\in \Z/n\Z$ (see
the beginning of this section).

We say that a Repeller/Attractor configuration of order $n\geq 3$ is an {\it elliptic configuration}\index{Repeller/Attractor configuration! elliptic} if :

\begin{enumerate}
\item the cyclic order of the sets $r_i, a_i$, $i\in \Z/n\Z$, satisfies the 
{\it elliptic order property}\index{elliptic order property}:
 $$a_0\to r_2\to a_1\to \ldots\to a_i\to r_{i+2}\to a_{i+1}\to \ldots\to a_{n-1}\to r_1\to a_0.$$
\item for all $i\in\Z/n\Z$ there exists a brick $b_i\in R_i$ such that $b_{i_{\geq}}\cap A_i\neq \emptyset $;

\end{enumerate}

 We say that a Repeller/Attractor configuration is a {\it hyperbolic configuration}\index{Repeller/Attractor configuration! hyperbolic} if:
\begin{enumerate}
\item  the cyclic order of the sets $r_i, a_i$, $i\in \Z/n\Z$, satisfies the
 {\it hyperbolic order property}\index{hyperbolic order property}:
$$r_0\to a_0\to r_{1}\to a_1\to \ldots\to r_i\to a_{i}\to r_{i+1}\to a_{i+1}\to\ldots \to r_{n-1}\to 
 a_{n-1}\to r_0.$$ 
 \item for all $i\in\Z/n\Z$ there exists two bricks $b_i^i, b_{i}^{i-1}\in R_i$ such that $[b]_{i_>}^i\cap A_i\neq \emptyset$, and 
$[b]_{i_>}^{i-1}\cap A_{i-1}\neq \emptyset$;

\end{enumerate}

\begin{figure}[h]   \label{noorientado}
\begin{center}
\psfrag{a}{$R_0$}\psfrag{b}{$A_0$}\psfrag{c}{$R_1$}\psfrag{d}{$A_1$}\psfrag{e}{$R_2$}\psfrag{f}{$A_2$}\psfrag{m}{$R_3$}
\psfrag{n}{$A_3$}\psfrag{g}{$R_1$}\psfrag{h}{$A_0$}\psfrag{i}{$R_2$}\psfrag{j}{$A_1$}\psfrag{k}{$R_0$}\psfrag{l}{$A_2$}

   \subfigure[An elliptic configuration]{\includegraphics[scale=0.3]{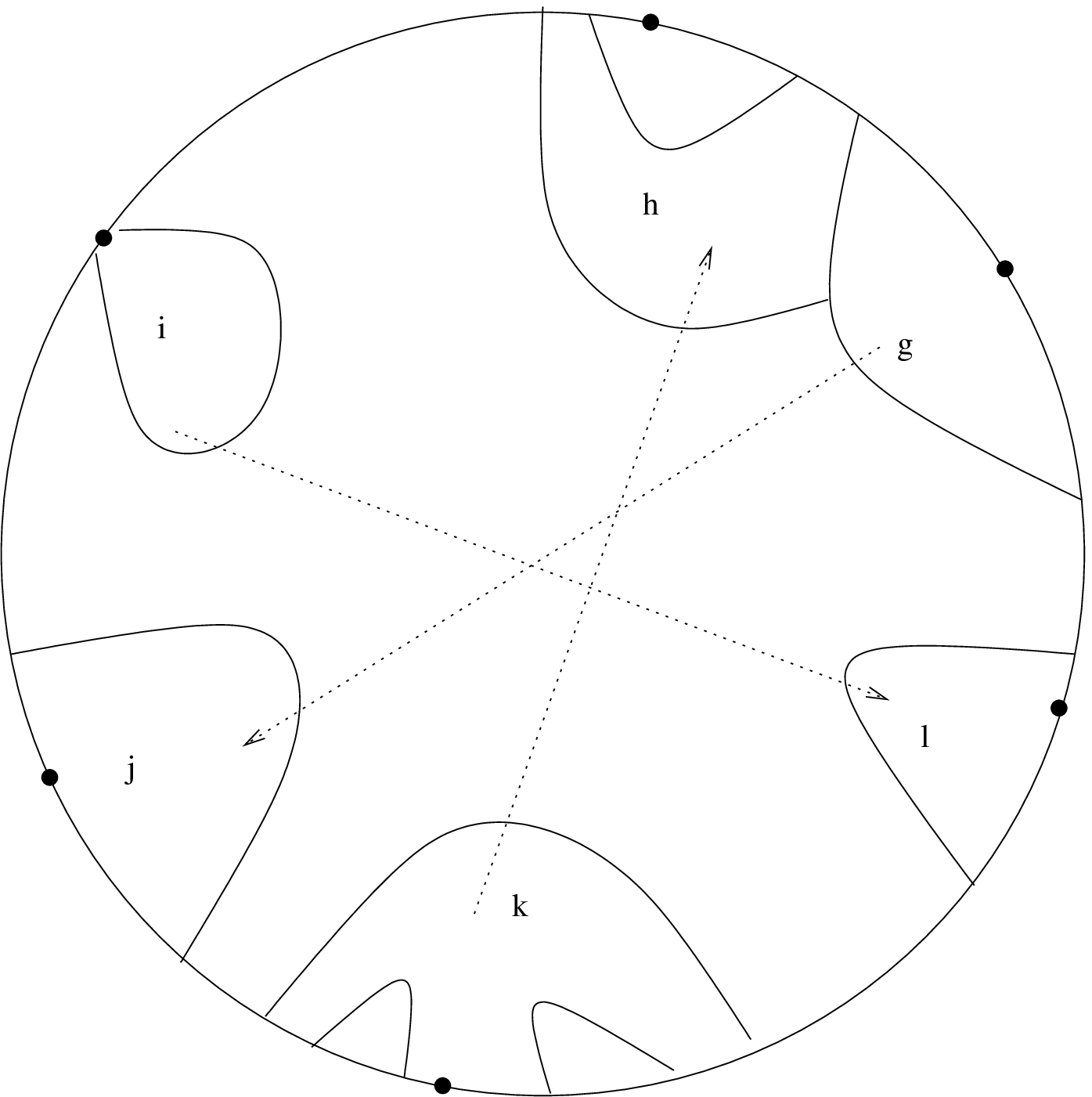}}\hspace{.25in}
    \subfigure[ A hyperbolic configuration] {\includegraphics[scale=0.3]{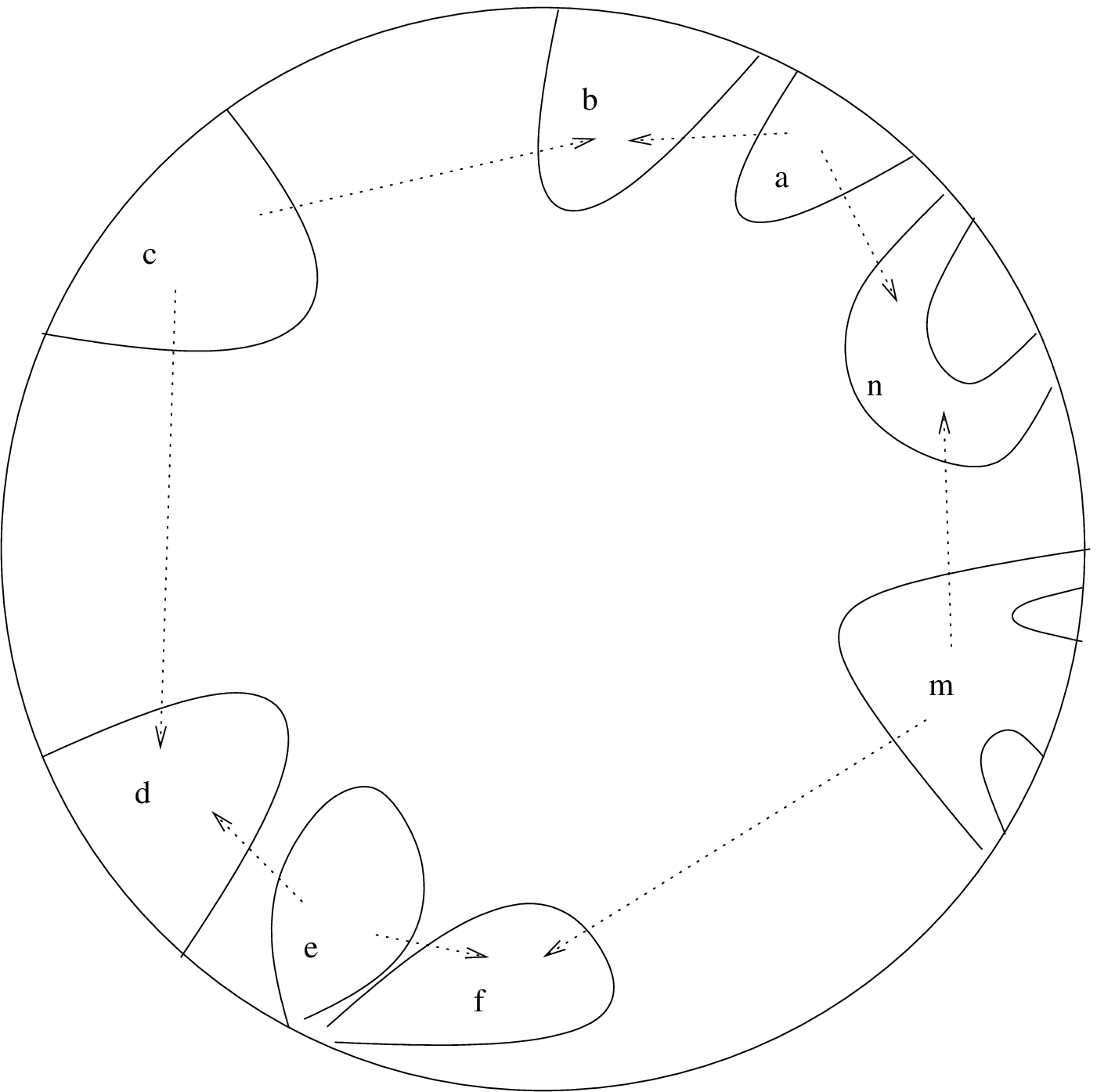}}\hspace{.25in} \\

\end{center}
 \end{figure}

We will show:

\begin{prop}\label{hc} If there exists an elliptic configuration of order $n\geq 3$, then $f$ is recurrent.
 
\end{prop}

\begin{prop}\label{2c} If there exists a hyperbolic configuration of order $n\geq 2$, then $\fix (f)\neq \emptyset$.
 
\end{prop}

One could think that Proposition \ref{2c} should give a negative-index fixed point, as the example that comes to mind is that
of a saddle point (see the figure below).\newpage

\begin{figure}[h]\label{silla}
\begin{center}
\psfrag{a}{$R_0 $}\psfrag{b}{$A_0$}\psfrag{c}{$R_1$}\psfrag{d}{$A_1$}
{\includegraphics[scale=0.6]{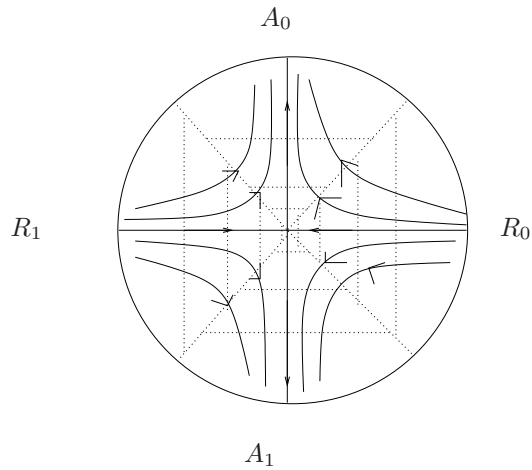}}
\caption{A hyperbolic configuration arising from a saddle point.}
\end{center}
\end{figure}

However, this is not the case, as the following example shows. 

\begin{ex}  Let $f_1$ be the time-one map of the flow whose orbits are drawn in the following figure:
 
\begin{figure}[h]
\begin{center}
\psfrag{e}{$R_0 $}\psfrag{b}{$A_0$}\psfrag{c}{$R_1$}\psfrag{d}{$A_1$}
{\includegraphics[scale=0.6]{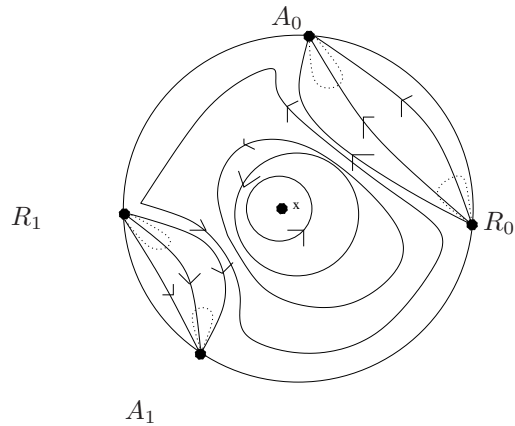}}
\caption{A hyperbolic configuration without a fixed point of negative index.}
\end{center}
\end{figure}

One can perturb $f_1$ in a homeomorphism $f$ such that:

\begin{enumerate}
\item $\fix(f) = \fix(f_1)= \{x\}$,
\item $f = f_1$ in a neighbourhood of $x$,
\item $f = f_1$ in a neighbourhood of $S^1$ (and so $f$ preserves the repellers and attractors drawn in dotted lines),
\item there is an $f$-orbit from $R_0$ to $A_1$,
\item there is an $f$-orbit from $R_1$ to $A_0$.

\end{enumerate}

So, $((R_i)_{i\in\Z/2\Z}, (A_i)_{i\in\Z/2\Z})$ is a hyperbolic configuration for $f$, but the only fixed point
$f$ has is an index-one fixed point.

\end{ex}

We define an order relationship in the set of Repeller/Attractor configurations of order $n$ :

$$((R_i)_{i\in \Z/n\Z} , (A_i)_{i\in \Z/n\Z}) \leq ((R_i')_{i\in \Z/n\Z} , (A_i')_{i\in \Z/n\Z})$$\noindent if and only if for all 
$i \in \Z/n\Z$ $$A_i \subseteq A_i ' \textrm{ and}\  R_i \subseteq R_i '.$$   \\

As the union of attractors (resp. repellers) is an attractor (resp. repeller), the existence of an 
elliptic (resp. hyperbolic) Repeller/Attractor configuration implies the existence
of a maximal elliptic (resp.hyperbolic) Repeller/Attractor configuration by Zorn's lemma.

\begin{ex} The hyperbolic configuration in Figure 6 is maximal.
 
\end{ex}

\noindent {\bf We will assume for the rest of this section that $f$ is non-recurrent.}  In particular, for any brick $b\in B$,
the sets
$[b]_\geq$, $[b]_>$, $[b]_\leq$ and $[b]_<$ are connected (see Proposition \ref{futcon}).\\

The following lemma is an  immediate consequence of the maximality of configurations:

\begin{lema}\label{max} Let $((R_i)_{i\in \Z/n\Z} , (A_i)_{i\in \Z/n\Z})$ be a maximal configuration 
(either elliptic or hyperbolic), and consider a brick $b \in B\backslash \cup_{i\in \Z/n\Z} (R_i\cup A_i)$.
If  $b$ is adjacent to $R_i$, then there exists, $j\neq i$, such that  $[b]_<\cap R_j \neq \emptyset$
in $B$.  If  $b$ is adjacent to $A_i$, then there exists, $j\neq i$, such that  $[b]_>\cap A_j \neq \emptyset$
in $B$.
 
\end{lema}

\begin{proof} Let $b\in B \backslash \cup_{i\in \Z/n\Z} (R_i\cup A_i)$ be adjacent to $R_i$. As both $R_i$ and $[b]_\leq$ are 
connected and they intersect, it
follows that the repeller $R= [b]_\leq \cup R_i$ is connected.  As our configuration is maximal and $R_i \subsetneq R$, 
there exists  $X\in {\cal{E}}\backslash \{R_i\}$, such that $R\cap X \neq \emptyset$ (in $B$).  As the sets in ${\cal{E}}$ are pairwise
disjoint, and $b$ does not belong to $X$, this implies that $[b]_<\cap X\neq \emptyset$ (in $B$).  So, $X = R_j$ for some $j \neq i$
, because  $[b]_\leq$ 
cannot intersect any attractor (see Remark \ref{br}, item \ref{br2}).
The second statement in the lemma is proved analogously.

\end{proof}

We say that a brick $b \in B$ is a {\it connexion brick}\index{brick! connexion} from  $R_j$ to $A_j$  if:
\begin {enumerate}
 \item $b \in B\backslash \cup_{i\in \Z/n\Z} (R_i\cup A_i)$,
\item $b$ is adjacent to $R_j$ and
\item $[b]_>$ contains a brick $b' \in B \backslash  \cup _{i\in \Z/n\Z} (R_i \cup A_i)$ which is adjacent to $A_j$.\\

\end {enumerate}

\begin{figure}[h]\label{conbric}
\begin{center}
\psfrag{x}{$R_j$}\psfrag{y}{$A_j$}\psfrag{b}{$b$}\psfrag{c}{$b'$}\psfrag{d}{$\subset [b]_>$}
{\includegraphics[scale=0.25]{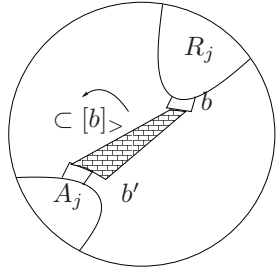}}
\caption{A connexion brick.}
\end{center}
\end{figure}

\begin{lema}\label{ady}  Let $((R_i)_{i\in \Z/n\Z} , (A_i)_{i\in \Z/n\Z})$ be a maximal elliptic or hyperbolic 
configuration. The following two conditions guarantee the existence of a connexion brick from $R_i$ to $A_i$:

\begin{enumerate}
\item\label{adj2} There exists a brick $b\notin \cup_{i\in \Z/n\Z} (R_i\cup A_i)$ which is adjacent to both $R_i$ and $A_i$,
 \item \label{notadj} $R_i$ is not adjacent to $A_i$.
\end{enumerate}

\end{lema}

\begin{proof} \ref{adj2}. Let $b'\notin \cup_{i\in \Z/n\Z} (R_i\cup A_i)$ be adjacent to both $R_i$ and $A_i$. As a subset of
$B$, the repeller $[b']_<$ meets a repeller $R_j$ different from $R_i$ (Lemma \ref{max}), meets $R_i$ because
$b'$ is adjacent to $R_i$ (Remark \ref{br}, item \ref{br3}), and does not meet any $A_j$, $j\in \Z/n\Z$ (Remark \ref{br}, 
item \ref{br2}). As it is connected, $[b']_<$ contains
a brick $b$ which is adjacent to $R_i$, which implies that $b \notin \cup_{i\in \Z/n\Z} (R_i\cup A_i)$
(Remark \ref{br}, item \ref{br4}).
As $b'\in [b]_>$, and $b'$ is adjacent to $A_i$, $b$ is a connexion brick from $R_i$ to $A_i$.

\ref{notadj}. Assume that $R_i$ is not adjacent to $A_i$. We know there exists $b_i\in R_i$ such that
$ [b_i]_\geq\cap A_i\neq \emptyset$. As $[b_i]_\geq$ is connected, it  contains a brick $b'$ adjacent to $A_i$. This 
brick $b'$ is not contained in $R_i$; otherwise, $R_i$ would be adjacent to $A_i$. Neither it is contained in any attractor or 
in any  repeller other that $R_i$ (Remark \ref{br}, items \ref{br2} and \ref{br4}). Therefore,  
$b'\notin \cup_{i\in \Z/n\Z} (R_i\cup A_i)$ .

As $b_i\in [b']_\leq$ and $[b']_\leq$ is connected, $[b']_\leq$  contains a brick $b$ adjacent to $R_i$. If $b\in [b']_<$, 
then $b$ is a connexion brick from $R_i$ to $A_i$ (again, $b\notin \cup_{i\in \Z/n\Z} (R_i\cup A_i)$ by Remark \ref{br}, items
\ref{br2} and \ref{br4}). If $b=b'$, then $b$ is adjacent to both $R_i$ and $A_i$ and we are done by the previous item.

\end{proof}

\begin{obs} Connexion bricks do not always exist; figure 6 exhibits an example. Of course, none of the conditions of Lemma \ref{ady}
is satisfied.  Indeed, in this example $\cup_{i\in \Z/2\Z} (R_i\cup A_i) = B$ and $R_i$ is adjacent to $A_i$ for all 
$i\in \Z/2\Z$.

\end{obs}

\subsection{The elliptic case.}

The following consequences of the elliptic order property will be used in the proof of Proposition \ref{hc}:

\begin{lema}\label{eop} Let $((R_i)_{i\in \Z/n\Z} , (A_i)_{i\in \Z/n\Z})$ be an elliptic configuration.

\begin{enumerate}
 \item\label{eop1} If $C\subset B$ is a connected set containing both $R_i$ and $A_i$,  and $C\cap (R_{i+1}\cup A_{i+1}) =
\emptyset$ in $B$,
then  $R_{i+1}$ and $A_{i+1}$ belong to different connected components of 
$\D\backslash \inte (C)$; in particular $R_{i+1}\cap A_{i+1} = \emptyset$ in $\D$.

\item \label{eop2} If $C\subset B$ is a connected set containing both $R_i$ and $R_{i+1}$, and $C\cap 
(R_{i-1}\cup A_{i-1})=\emptyset$ in $B$,
then $R_{i-1}$ and 
$A_{i-1}$ belong to different connected components of $\D\backslash \inte (C)$; in particular 
$R_{i-1}\cap A_{i-1} = \emptyset$ in $\D$.

\item \label{eop3} If $C\subset B$ is a connected set containing every repeller $R_i$, and disjoint (in $B$) from every
attractor $A_i$, then the $n$ attractors $\{A_i\}$ belong to $n$ different connected components of $\D\backslash \inte (C)$.
\end{enumerate}

\end{lema}

\begin{proof}
\begin{enumerate}
 \item First we remark that $C\cap (R_{i+1}\cup A_{i+1}) =
\emptyset$ in $B$ implies $\inte (R_{i+1}) \cap \inte (C) = \emptyset$ and $\inte (A_{i+1}) \cap \inte (C) = \emptyset$. Besides, 
$\inte (C)$ is  a connected set  containing both $r_i$ and $a_i$. So, the elliptic order property implies
that $r_{i+1}$ and $a_{i+1}$ belong  to different connected components of $\D\backslash \inte (C)$. Now, $\inte (R_{i+1})$ and 
$\inte (A_{i+1})$ belong to different connected components of $\D\backslash \inte (C)$.  As each connected component of  
$\D\backslash \inte (C)$ is closed (in $\D$), we obtain that $R_{i+1}$ and $A_{i+1}$ belong to different connected components of 
$\D\backslash \inte (C)$; in particular $R_{i+1}\cap A_{i+1} = \emptyset$ in $\D$.

\item As before, we know that  $\inte (R_{i-1}) \cap \inte (C) = \emptyset$ and $\inte (A_{i-1}) \cap \inte (C) = \emptyset$.
Besides,  
$ \inte (C)$ is a connected set containing both $r_i$ and $r_{i+1}$.  So, the elliptic order property implies
that $r_{i-1}$ and $a_{i-1}$ belong  to different connected components of $\D\backslash \inte(C)$. It follows that $\inte (R_{i-1})$ and 
$\inte (A_{i-1})$ belong to different connected components of $\D\backslash \inte (C)$, and we conclude as in the preceding item.

\item As before, we know that  $\inte (A_i) \cap \inte (C) = \emptyset$ for all $i\in \Z/n\Z$. Furthermore, $ \inte (C)$ is a 
connected set containing $r_i$ for all $i\in \Z/n\Z$.  So, the elliptic order property implies
that each $a_i$, $i\in \Z/n\Z$ belong to a different connected component of $\D\backslash \inte (C)$. It follows that each 
$\inte (A_i)$, 
$i\in \Z/n\Z$, belong to a different connected component of $\D\backslash \inte (C)$, and we conclude as in the preceding item.

\end{enumerate}

\end{proof}

\begin{lema}\label{bcon} Let $((R_i)_{i\in \Z/n\Z} , (A_i)_{i\in \Z/n\Z})$ be a maximal elliptic configuration. Then, for
 some $i \in \Z / n\Z$ there exists a connexion brick from $R_i$ to $A_i$.
 
\end{lema}

\begin{proof} Because of lemma \ref{ady}, it is enough to show that for some $i \in \Z / n\Z$,  $R_i$ is not adjacent to
$A_i$.

If $R_i$ is adjacent to $A_i$, then $C = R_i\cup A_i $ is a connected set containing $R_i$ and $A_i$.  Besides,  
$C\cap (R_{i+1}\cup A_{i+1}) =\emptyset$ in $B$, because the sets in ${\cal E}$ are pairwise disjoint.  So, item \ref{eop1} of the preceeding lemma tells us that
$R_{i+1}\cap A_{i+1} = \emptyset$ in $\D$. In particular, $R_{i+1}$ cannot be adjacent to $A_{i+1}$.  
\end{proof}

The following lemma tells us that it is enough to prove Proposition \ref{hc} for configurations of order $n = 3$:

\begin{lema}  Let $((R_i)_{i\in \Z/n\Z} , (A_i)_{i\in \Z/n\Z})$ be an elliptic configuration of order $n>3$. Then, there
exists an elliptic configuration $((R'_i)_{i\in \Z/(n-1)\Z} , (A'_i)_{i\in \Z/(n-1)\Z})$ of order $n-1$.
\end{lema}

\begin{proof} We claim that there exists a brick $b\in R_0$ such that $[b]_\geq \cap A_1\neq \emptyset$. Indeed, 
$$(R_0 \cup [b_0]_\geq\cup A_0)\cap R_1 = \emptyset \textrm{ in } B,$$ \noindent by Remark \ref{br}, item \ref{br2} (we recall
that for all $i\in\Z/n\Z$ there exists $b_i\in R_i$ such that $[b_i]_\geq \cap A_i\neq \emptyset$).  So,
 Lemma 
\ref{eop}, item \ref{eop1}  implies that either $$(R_0 \cup [b_0]_\geq
\cup A_0)\cap A_1\neq \emptyset \textrm{\ in \ } B,$$\noindent or $\inte (R_0 \cup [b_0]_\geq\cup A_0)$ separates $R_1$ from 
$A_1$ (recall that $b_0\in R_0, [b_0]_{\geq}\cap A_0 \neq \emptyset$,  and that the future of any brick is connected).
In the first case, necessarily $$ [b_0]_{\geq}\cap A_1\neq  \emptyset \textrm{\ in \ } B,$$ \noindent and we take $b = b_0$.
In the second case, we obtain $$(R_0 \cup [b_0^-]_{\geq} \cup A_0)\cap (R_1\cup [b_1^+]_{\leq}\cup A_1)\neq
\emptyset  \textrm{\ in \ } B,$$ \noindent where $b_{1}^+\in [b_1]_{\geq}\cap A_1$ .  By Remark \ref{br}, item \ref{br2}, we
 know that $[b_0]_{\geq} \cap R_1 = \emptyset$ and 
$[b_1^+]_{\leq} \cap A_0 = \emptyset$. So, in fact 
$$(R_0 \cup [b_0]_{\geq} )\cap ([b_1^+]_{\leq}\cup A_1)\neq \emptyset  \textrm{\ in \ } B.$$ 
If $R_0\cap [b_1^+]_{\leq}\neq \emptyset$ in $B$, we
take any brick $b\in R_0\cap [b_1^+]_{\leq}$; if $[b_0]_{\geq}\cap ([b_1^+]_{\leq} \cup A_1)\neq \emptyset$ in $B$, we take 
$b = b_0$. (Note that $b\in [b_1^+]_{\leq}$ implies $b_{1}^+\in  [b]_\geq\cap A_1$). This finishes the proof of our claim.

Now, by defining $$R'_0 = R_0, \ R'_i = R_{i+1}\textrm{ for } 1\leq i\leq n-2,$$ 
$$A'_i = A_{i+1} \textrm{ for } 0\leq i \leq n-2,$$ we are done.
 
\end{proof}

We are now ready to prove Proposition \ref{hc} :

\begin{proof} Because of the previous lemma, we can suppose that there exists an elliptic configuration of order $n=3$ and  
take a maximal one $$((R_i)_{i\in \Z/3\Z} , (A_i)_{i\in \Z/3\Z}).$$ \noindent We will show that our assumption that $f$ is not recurrent
contradicts the maximality of this configuration.  Lemma \ref{bcon} allows us to consider a connexion brick
$b$ from $R_i$ to $A_i$, for some $i\in \Z/3\Z$, and there is no loss of generality in supposing $i = 0$. Let $b' \in B 
\backslash \cup_{i\in \Z/3\Z} (R_i\cup A_i)$ be adjacent to $A_0$ and such that $b' \in [b]_>$. We will first show that
$[b]_<$ meets every repeller and no attractor in the configuration.  Then, by defining $A'_i$ as to be the connected component of
$B\backslash (\cup _{i\in \Z/3\Z} R_i\cup [b]_<)$ containing $A_i$, we will be able to show that 
$((R_i)_{i\in \Z/3\Z}, (A'_i)_{i\in \Z/3\Z})$ is an elliptic configuration strictly bigger than the initial configuration, due
to the fact that $b'\in A'_0\backslash A_0$.

Indeed, we know by Lemma \ref{max} that $[b]_\leq 
\cap R_{j}\neq \emptyset$ for some $j\in \{1,2\}$.  We 
will suppose $[b]_\leq \cap R_{1}\neq \emptyset$; the proof is analogous in the other case. We claim that this implies 
$[b]_\leq \cap R_{2} \neq \emptyset$.  To see this, note that item \ref{eop2} of Lemma \ref{eop} implies
$$R\cap (R_2\cup [b_2]_{\geq}\cup A_2)\neq \emptyset,$$\noindent where
$$R= R_0\cup [b]_\leq \cup R_{1}.$$
 So, actually  
$$[b]_\leq\cap [b_2]_\geq \neq \emptyset,$$  \noindent which implies $[b]_\leq \cap R_{2} \neq \emptyset$.

We have obtained that $R'= \cup _{i\in \Z/3\Z} R_i \cup [b]_\leq$ is a connected repeller disjoint (in $B$)
from every attractor $A_i$, $i\in \Z/3\Z$ (Remark \ref{br}, item \ref{br2}).   Let  $A_j '$ be the connected component of
$B\backslash R'$ containing $A_j$ for all $j\in \Z/3\Z$. Then, the  sets  $A_j '$ $j\in \Z/3\Z$ are pairwise disjoint (in $\D$) by the 
elliptic order property. We know that $b'\in B\backslash R'$; otherwise, we would have $b'\in [b]_\leq$ as $b' \notin \cup_{i\in \Z/3\Z}
(R_i\cup A_i)$, which is impossible because $b' \in [b]_>$ and we are supposing that $f$ is non-recurrent. So, $A_0$ is strictly 
contained in $A_0 '$ and we deduce that  $((R_i)_{i\in \Z/n\Z} , (A_i ')_{i\in \Z/3\Z})$ is 
an elliptic configuration strictly greater than $((R_i)_{i\in \Z/3\Z} , (A_i)_{i\in \Z/3\Z})$, contradicting the maximality
of the configuration.

\end{proof}

\subsection{The hyperbolic case.}

In what follows, we deal with the hyperbolic case. The proof of the following lemma is analogous to that of Lemma \ref{eop},
substituting of course the elliptic order property by the hyperbolic order property.

\begin{lema}\label{hop} Let $((R_i)_{i\in \Z/n\Z} , (A_i)_{i\in \Z/n\Z})$ be a hyperbolic configuration.

If $C\subset B$ is a connected set containing $R_i$ and $R_{i+1}$, and $C\cap A_m=\emptyset$ in
$B$ for all $m\in \Z/n\Z$, then $\inte (C)$ separates (in $\D$) $A_i$ from any $A_j$, $j\neq i$.

\end{lema}

\begin{lema}\label{upsi} Let $((R_i)_{i\in \Z/n\Z} , (A_i)_{i\in \Z/n\Z})$ be a  hyperbolic configuration.  If 
$X\in {\cal E}$, then there is only one connected component of $B\backslash X$ containing sets in  
${\cal E}$.
\end{lema}

\begin{proof}  We will suppose that  $X= R_j, j\in \Z/n\Z$; the proof is analogous for any $X\in {\cal E}$.
 We will show that the connected component $C$ of $B\backslash R_j$ containing $A_j$ contains every $X\in {\cal E}, 
X\neq R_j$. As $B\backslash R_j$ is an attractor, and there is
a brick in $R_{j+1}$ whose (connected) future intersects
$A_j$, we have that $R_{j+1}\subset C$ (we recall that every connected component of an attractor is an attractor, 
see Proposition \ref{futcon}). As there is also a brick in $R_{j+1}$ whose future intersects
$A_{j+1}$, the same argument shows that $A_{j+1}\in C$. By induction, we get that every $X\in {\cal E}\backslash \{R_j\}$
belongs to $C$.
\end{proof}

\begin{lema}\label{bconh}   Let $((R_i)_{i\in \Z/n\Z} , (A_i)_{i\in \Z/n\Z})$ be a  maximal hyperbolic configuration. One
of the following is true:

\begin{enumerate}
 \item $\fix(f)\neq \emptyset$,
\item there exists a connexion brick from $R_j$ to $A_j$ for some $j\in \Z/n\Z$.
\end{enumerate}
 
\end{lema}

\begin{proof} We will show that if $\fix(f) = \emptyset$, then there exists a connexion brick from $R_j$ to $A_j$ 
for some $j\in \Z/n\Z$. By Lemma \ref{ady}, we
can suppose that $R_i$ is adjacent to $A_i$ for all $i\in \Z/n\Z$. If  $R_i$ is adjacent to $A_i$, either there is one 
connected component $\gamma$ of $\partial R_i$ which is also a connected component of $\partial A_i$ or there is a point 
$x\in R_i\cap A_i\cap \partial (R_i\cup A_i)$. If $\fix (f) = \emptyset$, then every connected component of
$\partial X$ is an embedded line in $\D$, for any $X\in {\cal E}$.  So, if there were one 
connected component $\gamma$ of $\partial R_i$ which is also a connected component of $\partial A_i$, $\gamma$ would 
separate $\D$ into two connected components $C_1$ and $C_2$, containing $\inte (A_i)$ and $\inte (R_i)$ respectively.  Then,
Lemma \ref{upsi} would imply that every set in  ${\cal E}\backslash R_i$ belongs to $C_1$, and that every set in  ${\cal E}
\backslash A_i$ belongs to $C_2$, which is clearly impossible.  

We are left with the case where there is a point $x\in R_i\cap A_i\cap \partial (R_i\cup A_i)$. This point $x$ is necessarily
a vertex of $\Sigma ({\cal D})$.  It belongs to three bricks: one that belongs to $R_i$, another one which belongs to $A_i$,
and a third one which is adjacent to both $R_i$ and $A_i$. This third brick brick does not belong to any repeller or 
attractor, as it is adjacent to both $R_i$ and $A_i$ (see Remark \ref{br}, item \ref{br4}). So,  by Lemma \ref{ady}, item 
\ref{adj2}, there exists a connexion brick from $R_i$ to $A_i$.
 
\end{proof}

We will prove  Proposition \ref{2c} by induction on the order of the configuration. We begin by the case $n = 2$:

\begin{prop} If there exists a hyperbolic configuration of order $2$, then $\fix (f)\neq\emptyset$.
 
\end{prop}

\begin{proof}  Suppose there exists such a  configuration  and take a maximal one $$((R_i)_{i\in \Z/2\Z} , 
(A_i)_{i\in \Z/2\Z}).$$
Because of Lemma \ref{bconh},   we can suppose that there exists a connexion brick $b$ from $R_j$ to
$A_j$ for some $j\in \Z/2\Z$, and there is no loss of generality in supposing $j=0$.  We take a brick $b'$ 
such that $b' \in [b]_>$, $b'\in B \backslash \cup_{i\in \Z/n\Z} (R_i\cup A_i)$ and
 $b'$ is adjacent to $A_0$. Here again, we will first show that $[b]_<$, the strict past of $b$, meets every repeller
and no attractor in the configuration. Then, by
defining $A'_i$ as the connected component of $B\backslash (\cup_{i\in \Z/2\Z}R_i\cup [b]_<)$ containing $A_i$, we will
be able to show that $((R_i)_{i\in \Z/2\Z} , (A'_i)_{i\in \Z/2\Z})$ is a hyperbolic configuration strictly greater than the 
original one, due to the fact that $b'\in A'_0\backslash A_0$. 

Because of Lemma \ref{max} we know that $[b]_<\cap R_1\neq \emptyset$ in $B$.
So,  $$R = R_0\cup b_{{\leq}}\cup R_1$$
\noindent is connected and disjoint from every attractor in the configuration (see Remark \ref{br}, item \ref{br2}). 
It follows that $\inte (R)$ separates  $A_0$ from $A_1$, this being the content of Lemma \ref{hop}. Let $A'_i$ be the
connected component of $B\backslash R$ containing $A_i$, $i\in \Z/2\Z$.  Then, $A'_0\cap A'_1 = \emptyset$. We know
that $b'\notin R$, because $b'\in [b]_>$, and otherwise $f$ would be recurrent. So, 
$b'$ belongs to $A'_0\backslash A_0$, contradicting the maximality of $((R_i)_{i\in \Z/2\Z} , (A_i)_{i\in \Z/2\Z})$.

\end{proof}

Now we are ready to prove  Proposition \ref{2c}:

\begin{proof}
We will show that given a maximal hyperbolic configuration of order $n>2$ $$((R_i)_{i\in \Z/n\Z} , (A_i)_{i\in \Z/n\Z}),$$ 
\noindent we can
 construct a new hyperbolic
 configuration whose order is strictly smaller than $n$ (and yet greater or equal to $2$). We can suppose there exists a 
connexion brick $b$ from $R_0$ to $A_0$.  We take a brick $b'\in [b]_>$ such that $b'\in B 
\backslash \cup_{i\in \Z/n\Z} (R_i\cup A_i)$ and $b'$ is adjacent to $A_0$.  By Lemma \ref{max}, 
$$[b]_\leq\cap R_i\neq \emptyset \textrm{ for some }\ i \neq 0. $$ We can suppose that 
$i\neq 1$; otherwise, we could use the same argument we used for the case $n=2$. Indeed, Lemma \ref{hop}
would imply that $R_0\cup R_1\cup [b]_\leq$ is a connected repeller which separates $A_0$ from any other
$A_j$, $j\neq 0$. So, by replacing $A_0$ by $A'_0$, the connected component of 
$B\backslash (R_0\cup R_1\cup [b]_\leq)$ containing $A_0$, we would have a hyperbolic
configuration strictly bigger than the original one.

So, we may suppose that $$i = \min \{j\in \{1, \ldots, n-1\}: [b]_\leq\cap R_j\neq \emptyset \}\neq 1 .$$ 
\noindent We define $$R = R_0 \cup  [b]_\leq\cup R_i,$$ \noindent which is a connected repeller. 

If we set $R'_0 = R$, $R'_j = R_j$ for all $1\leq j\leq i-1$, and $A'_j=A_j$ for all $i\in \Z/n\Z$, $0\leq j \leq i-1$.  
Then, $((R'_j)_{j\in \Z/i\Z}, (A'_j)_{j\in \Z/i\Z})$ is a hyperbolic configuration of order $i$, $2\leq i <n$.

\end{proof}

\subsection{Proof of the Theorem}

In this section we prove Theorem \ref{main}.
We fix an orientation preserving homeomorphism $f :\D\to \D$ which realizes a compact convex polygon $P\subset \D$, and
can be extended to a homeomorphism of $\D\cup (\cup_{i\in\Z/n\Z}\{\alpha_i, \om_i\})$.  We suppose that $i(P)\neq 0$, and we will show  that either $f$
is recurrent, or we can construct an elliptic  or hyperbolic Repeller/Attractor configuration.\\

Some polygons can be simplified, due to the fact that they may have  ``extra'' edges. More precisely,  
we will say that the polygon $P$ is minimal if for every $i\in \Z/n\Z$, the lines $\{\Delta _j: j\neq i\}$
do not bound a compact convex polygon.  The following lemma tells us that it is enough to deal with minimal polygons.

\begin{lema}\label{min} The map $f$ realizes a minimal polygon $P'$ such that $i(P') = i(P)$, or a triangle $T$ such that $i(T)=1$.

\end{lema}

\begin{figure}[h]
\begin{center}

{\includegraphics[scale=0.6]{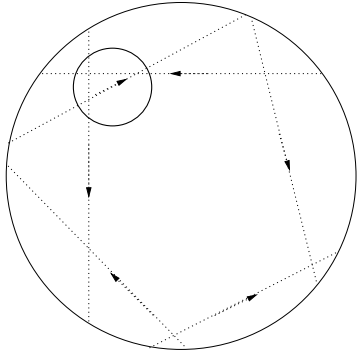}}
\caption{A non-minimal hexagon of index $-2$ presenting an index $1$ subtriangle.}
\end{center}
\end{figure}

\begin{proof}  If $P$ is not minimal, then there exists $i\in \Z/n\Z$ such that the straight lines 
$\{\Delta _j: j\neq i\}$ bound a compact polygon $P' \subset \D$.  The line $\Delta _i$
 intersects in $\D$ both $\Delta _{i-1}$ and $\Delta _{i+1}$; it follows that necessarily $$\Delta _{i-1}\cap 
\Delta _{i+1}\cap \D \neq \emptyset.$$ 

So, the lines  $\Delta _{i-1}$, $\Delta _i$ and $\Delta _{i+1}$ bound a triangle $T\subset \D$.  Moreover, 
$$i(P') = i(P) + i(T), $$ \noindent and the only possibilities for the index of a triangle are $0$ or $1$. 

If $i(T) = 1$, we are done.  Otherwise, $i(P') = i(P)$.  If $P'$ is minimal, we are done. If not, we apply the same procedure 
as before.  We continue like this until we obtain an index $1$ triangle, or a minimal polygon with the same index as $P$.
\end {proof}

Let us state our first proposition:

\begin{prop}\label{ell} If $i(P)=1$, then $f$ is recurrent.
 
\end{prop}

We observe that  lemma \ref{min} allows us to suppose that $P$ is minimal; we will also suppose that the boundary of $P$ is 
positively oriented. With these assumptions,  the order of the points $\alpha _i, \om _i, i\in \Z/n\Z $ at the circle at 
infinity satisfies:

$\om_0\to \alpha_2\to \om_1\to \ldots\to \om_i\to \alpha_{i+2}\to \om_{i+1}\to \ldots\to \om_{n-1}\to \alpha_1\to \om_0.$

{\bf From now on, we suppose that $f$ is not recurrent.} We apply Lemma \ref{ends} and obtain a family of closed disks 
$(b_i'^l)_{l\in \Z\backslash \{0\}, i\in \Z/n\Z}$.
The hypothesis on the points $ \alpha _i, \om _i,\  i\in \Z/n\Z ,$ allows us to suppose that all the disks 
$(b_i'^l)_{l\in \Z\backslash \{0\}, i\in \Z/n\Z}$ have
pairwise disjoint interiors (see Remark \ref{ptsdf}). \\

\begin{obs}\label{airi} The sets $\Gamma _i^-\cap \D,\  \Gamma _i^+\cap \D$ defined in Remark \ref{puntas} satisfy 
the elliptic order property (see Remark \ref{order}).
\end{obs}

By Corollary \ref{bon}, we can construct a free brick decomposition 
$(V,E,B)$ such that for all $i\in \Z/n\Z$ and for all $l\in \Z\backslash \{0\}$, there 
exists $b_i^l\in B$ such that $b_i'^l\subset b_i^l$.  Moreover, one can suppose that this decomposition is maximal.

\begin{obs}\label{cruz} As $\cup _{l>0} [b_i^l]_\leq$ is a connected set whose closure contains both $\alpha _i$ and $\om _i$,
 if $\Gamma : [0,1] \to \overline \D$ is an arc that separates $\alpha _i$ from $\om _i$, then $\Gamma \cap 
\cup _{l>0} [b_i^l]_\leq\neq \emptyset$. 
 
\end{obs}

\begin{lema} If for some $k>0$, $m>0$ and $j\in \Z/n\Z$, both $b_j^k$ and $b_{j+1}^k$ are contained in $[b_i^{-m}]_>$,  
then there exists $l>0$ such that $b_{j+2}^l\in [b_i^{-m}]_> $.
 
\end{lema}

\begin{proof} If $b_j^k$ and $b_{j+1}^k$ are contained in $[b_i^{-m}]_>$, then $b_j^p$ and $b_{j+1}^p$
are contained in $[b_i^{-m}]_>$ for all $p\geq k$ (note that $[b_i^{-m}]_>$ is an attractor, and that Lemma \ref{ends}, item 6.
implies that $b_j^p\subset [b_j^k]_\geq$ for all $p\geq k$).  So, as $[b_i^{-m}]_>$ is connected, we can find an arc
$$\Gamma : [0,1] \to [b_i^{-m}]_> \cup \{\om_j, \om_{j+1}\}$$\noindent joining $\om_j$ and $\om_{j+1}$ 
(see Remark \ref{puntas}).  Then, 
$\Gamma$ separates $\alpha _{j+2}$ from $\om _{j+2}$ in $\overline \D$ (see Remark \ref{order}). By Remark \ref{cruz}, we 
obtain $$\Gamma \cap (\cup _{l>0}
[b_{j+2}]_<^l) \neq \emptyset.$$ \noindent So, $$ [b_i^{-m}]_> \cap (\cup _{l>0} [b_{j+2}]_<^l) \neq \emptyset , $$ from which one  
gets (as the future of any brick is an attractor) that there exists $l>0$ such that $b_{j+2}^l\in [b_i^{-m}]_> $.
 
\end{proof}

\begin{lema}\label{uno}{\bf (Domino effect)} There exists $k>0$ such that for all $i,j\in \Z/n\Z$, $[b_i^{-k}]_>$ contains $b_j^k$.
 
\end{lema}

\begin{proof} Fix $i\in \Z/n\Z$.  There exists an arc 
$$\Gamma : [0,1] \to \cup _{l>0} [b_i^{-l}]_> \cup \{\alpha _i, \om_{i}\}$$\noindent
joining $\alpha _i$ and $\om_{i}$ (see Remark \ref{puntas}). Then, $\Gamma$ separates
$\alpha _{i+1}$ from $\om _{i+1}$ in $\overline \D$ (see Remark \ref{order}). So, Remark \ref{cruz} gives us 
$$\Gamma \cap (\cup _{l>0} [b_{i+1}^{l}]_<) \neq \emptyset . $$  So, 
$$ (\cup _{l>0} [b_i^{-l}]_>)\cap (\cup _{l>0} [b_{i+1}^{l}]_<) \neq \emptyset , $$ from which one immediately gets 
that there exists $l_i, m_i>0$ such that $b_{i+1}^{l_i}\in [b_i^{-m_i}]_> $. 
As $b_{i}^{l_i}\in [b_i^{-m_i}]_> $ as well, the previous lemma tells us that there exists $l>0$ such
that $b_{i+2}^l\in [b_i^{-m_i}]_> $. We finish the proof of the lemma by induction, and then taking $k>0$ large enough. 
\end{proof}

We are now ready to prove Proposition \ref{ell}:

\begin{proof}  We will show that $(([b_i^{-k}]_<)_{i\in \Z/n\Z},([b_i^{k}]_>)_{i\in \Z/n\Z})$ is an elliptic configuration,
where $k>0$ is given by the preceding lemma.  This contradicts our assumption that $f$ is not recurrent, 
by Proposition \ref{hc}.

We define $r_i = \Gamma _i^-\cap \cup_{m\geq k} b_{i}^{-k}$, and $a_i = \Gamma _i^+\cap \cup_{m\geq k} b_{i}^{k}$, 
$i\in \Z/n\Z$;  we may  suppose that the sets $r_i$, $a_i$, 
$i\in \Z/n\Z$ are arcs (the sets $\Gamma _i^-\cap \D,\  \Gamma _i^+\cap \D$ were defined in Remark \ref{puntas}).
These arcs $a_i, r_i$, $i\in \Z/n\Z$  satisfy the elliptic order property (see Remark \ref{airi}).  Besides, for all $i\in
\Z/n\Z$,

\begin{itemize}
 \item $r_i\subset [b_i^{-k}]_<$, 
\item $a_i\subset [b_i^{k}]_>$, and
\item $b_i^k\in [b_i^{-k}]_>$.
\end{itemize}

So, we only have to show that the sets $\{[b_i^{-k}]_<\},\{[b_j^{k}]_>\},$ are pairwise disjoint.  As we are supposing that $f$ is not recurrent, the preceding
lemma gives us that for any pair of indices $i,j$ in $\Z/n\Z$: $$[b_i^{-k}]_<\cap [b_j^{k}]_> = \emptyset.$$ Let us show that
 for for any pair of different indices $i,j$ in $\Z/n\Z$ one has
$$ [b_i^{-k}]_<\cap [b_j^{-k}]_< = \emptyset.$$ \noindent  Otherwise, there would exist $i\neq j$ such that
$[b_i^{-k}]_<\cup [b_j^{-k}]_<$ is a connected set containing $r_i$ and $r_j$.  As $[b_i^{-k}]_>$ is a connected set 
containing $a_j$ for all $j\in \Z/n\Z$ (again by the preceding lemma), the elliptic order property tells us:
$$([b_i^{-k}]_<\cup [b_j^{-k}]_<)\cap [b_i^{-k}]_> \neq \emptyset.$$ \noindent We deduce (as $f$ is not recurrent) that
$$ [b_j^{-k}]_<\cap [b_i^{-k}]_> \neq \emptyset,$$\noindent but then $[b_j^{-k}]_<$ is a connected set containing both
$r_j$ and $r_i$, and once again the preceding lemma and the elliptic order property imply 
$$[b_j^{-k}]_<\cap [b_j^{-k}]_> \neq \emptyset,$$\noindent a contradiction.
To prove that for any pair of different indices $i,j$ in $\Z/n\Z$ one also has
$$ [b_i^{k}]_>\cap[b_j^{k}]_> = \emptyset, $$\noindent it is enough to interchange the roles of $<$ and $>$, $k$ and $-k$
in the proof we just did.
\end{proof}

Our next proposition finishes the proof of Theorem \ref{main}:

\begin{prop}\label{negip} If $i(P) < 0$, then $\fix (f) \neq
 \emptyset$.
 
\end{prop}

By Lemma \ref{min} and Proposition \ref{ell}, we can suppose that $P$ is minimal.  We would also like to suppose that 
$\delta _i = 1$ for all $i\in \Z/n\Z$, so as to fix the cyclic order of the points $\{\alpha _i\}, \{\om _i\},$ at 
the circle at infinity.  For
this reason, we introduce the following lemma. 

\begin{lema}\label{triangular} If $\delta _i = 0$ for some $i\in \Z/n\Z$, then there exists $g\in \homeo ^+ (\D)$ such that :
\begin{enumerate}
 \item $\fix (g) = \fix (f)$;
\item $g=f$ on the orbits of the points $z_j$, $j\notin \{i-1,i\}$,
\item there exists $z\in \D$ such that $\lim _{k \to -\infty} g^k(z) = \alpha_{i-1}$ and $\lim _{k \to +\infty} 
g^k(z) = \om _{i}$.
\end{enumerate}

\end{lema}

We will need the following lemma, which is nothing but an adaptation of Franks' Lemma (see 2.2).

\begin{lema}\label{gl} Let $(D_i)_{0\leq i \leq p}$ be a chain of free, open and pairwise disjoint disks for $f$, and take 
two points $x\in D_0$ and $y\in D_p$.

Then, there
exists $g\in \homeo ^+ (\D)$ and an integer $q\geq p$ such that:

\begin{itemize}
 \item $\fix(g) = \fix(f)$,
\item $g=f$ outside $\cup _{i=0}^p D_i$,
\item $g^q(x) = f(y)$.
\end{itemize}

\end{lema}

\begin{proof} Take $z_i\in D_i$ and  $k_i>0$ the smallest positive integer such that $f^{k_i} (z_i) \in D_{i+1}$, 
$i\in \{0, \ldots, p-1\}$. We may suppose that the chain $(D_i)_{0\leq i \leq p}$ is of minimal lenght; that is, every
$f^k(z_i), 0<k<k_i$ is outside $\cup _{j=0}^p D_j$.  We construct a homeomorphism $h_0$ which is the identity outside $D_0$
and such that $h_0(x) = z_0$, and a homeomorphism $h_p$ which is the identity outside $D_p$
and such that $h_p(f^{k_{p-1}}(z_{p-1})) = y$.  For $i\in \{1, \ldots, p-1\}$, we construct homeomorphisms $h_i$ such that:

\begin{itemize}
 \item $h_i$ is the identity outside $D_i$,
\item $h_i(f^{k_{i-1}}(z_{i-1})) = z_i$
\end{itemize}

Finally, we construct a homeomorphism $h$ which is the identity outside $\cup _{j=0}^p D_j$ and identical to $h_i$ in $D_i$,
$i\in \{0, \ldots, p\}$.

So, as the disks $\{D_i\}$ are free, $g = f\circ h$ satisfy all the conditions of the lemma. 
 
\end{proof}

The proof of Lemma \ref{triangular} follows.

\begin{proof} We will first construct a  brick decomposition that suits our purposes.  As the points $\alpha _{i-1},\alpha_i, 
\om_{i-1},  \om_i$ are all different and $f$ is not recurrent, we can construct families of closed disks 
$(b_i'^k)_{k\in \Z\backslash \{0\}}$, $(b_{i-1}'^k)_{k\in \Z\backslash \{0\}}$ as in Lemma \ref{ends} with the property that the interiors of the 
bricks in these families are pairwise disjoint.

Let $O= \cup _{i\in \Z/n\Z, k\in \Z} f^k (z_i)$.  Here again we construct a maximal free brick decomposition  such that for all $l \in \Z \backslash \{0\}$, there exists $b_i^l, b_{i-1}^l
\in B$ such that $b_i'^l\subset  b_{i}^l$ and $b_{i-1}'^l\subset b_{i-1}^l$.  Furthermore, we may suppose that  for all $x\in O$  there exists $b_x \in B$ 
such that $x \in \inte (b_x)$.

If $\delta _i = 0 $ for some $i \in \Z/n\Z$, then $P$ is either to the right of both $\Delta _i$ and $\Delta _{i-1}$ or either 
to the left of both $\Delta _i$ and $\Delta _{i-1}$.  We will suppose that $P$ is to the left of both
lines, as the other case is analogous. By Remark \ref{puntas}, we can find an arc 
$$\Gamma : [0,1]\to \cup _{l>0} [b_i^l]_<$$\noindent joining $\alpha _{i}$ and $\om_i$.  
So, $\Gamma$ separates in $\overline \D$ $\alpha _{i-1}$ from $\om _{i-1}$. This implies that there exist two positive integers
 $j,k $ such that $$[b_{i-1}^{-j}]_> \cap [b_i^{k}]_<\neq \emptyset$$\noindent (note that $\cup _{j>0} [b_{i-1}^{-j}]_>$ is a connected set whose
closure contains $\alpha_{i-1}$ and $\om _{i-1}$). So, we can find a sequence of
bricks $(b_m)_{0\leq m\leq p}$  such that $b_0 = b_{i-1}^{-j}$, $b_p =b_{i}^k$ and $f(b_m)\cap b_{m+1} \neq \emptyset$ if 
$m\in \{0,\ldots, p-1\}$.   We will suppose that this sequence is of minimal lenght, that is:

$$f(b_m)\cap b_{m'} \neq \emptyset \Rightarrow m' = m+1 (*).$$ \noindent

We define for all $1\leq m\leq p-1$  $$X_m = b_m \backslash O . $$\noindent
We also define  $$X_0 = b_0 \backslash (O -\{f^{-k_{i-1} -j+1}(z_{i-1})\})$$\noindent and 
$$X_p = b_p\backslash (O -\{f^{k_{i}
 + k-1}(z_{i})\})$$\noindent (we recall from Lemma \ref{ends} that $f^{-l_{i-1} -j+1}(z_{i-1})$ is the only point of the orbit 
of $z_{i-1}$  which lies in $b_0$, and that $f^{l_{i} + k-1}(z_{i})$ is the only point of the orbit 
of $z_{i}$ which lies in $b_p$). As every $x\in O$ belongs to the interior 
of a brick, we know that $$f(X_m)\cap X_{m+1} \neq \emptyset$$\noindent if 
$m\in \{0,\ldots, p-1\}$. 

For each $m\in \{0,\ldots, p-1\}$, we take $x_m\in X_m $ such that $f(x_m)\in X_{m+1}$.  We take an arc 
$\gamma _0\subset X_0$ from $f^{-k_{i-1} -j+1}(z_{i-1})$ to $x_0$, and an arc $\gamma _p\subset X_p$ from 
$f(x_{p-1})$ to $f^{k_{i} + k-1}(z_{i})$.  For each $m\in \{1, \ldots, p-1\}$ we take an
arc $\gamma _{m}\subset X_m$ joining $f(x_{m-1})$ and $x_m$.  As the interiors of the sets $\{X_m\}$  are pairwise
disjoint, the arcs $\{\gamma _m\}$ can only meet in their extremities.  However, condition $(*)$ implies that
the points $\{x_m\}$ (and thus the points $\{f(x_m)\}$  ) are all different.  Indeed, if
$x_m = x_{m'}$, then $f(x_m)\in X_{m'+1}$, and so $f(b_m)\cap b_{m'+1}\neq \emptyset$.  It follows by $(*)$ that
$m = m'$.  On the other hand, if $f(x_m) = x_{m'}$, we obtain that $f(b_m)\cap b_{m'}\neq \emptyset$, and so
$m' = m+1$.  This means that the arcs $\{\gamma _m\}$ are pairwise disjoint (some of them maybe reduced to a point).

It follows that we can thicken this
arcs $\{\gamma_m\}$ into free, open and pairwise disjoint disks $\{D_m\}$, such that $\gamma _m\subset D_m$, and such 
that $D_m\cap O = \emptyset$.

We are done by Lemma \ref{gl}.

\end{proof}

\begin{lema} Let $f$ realize a minimal $n$-gon $P$ such that $i(P)<0$. If $\delta _i =0$ for some $i\in \Z/n\Z$, then here
exists $g\in Homeo^+(\D)$ realizing an $n-1$-gone $P'$ such that $i(P') = i(P)$ and $\fix (g) = \fix(f)$.
\end{lema}

\begin{proof}   By Lemma \ref{triangular}, there exists $g\in \homeo ^+ (\D)$ such that :
\begin{enumerate}
 \item $\fix (g) = \fix (f)$;
\item $g=f$ on the orbits of the points $z_j$, $j\in \Z/n\Z$, $j\notin \{i-1,i\}$,
\item there exists $z\in \D$ such that $\lim _{k \to -\infty} g^k(z) = \alpha_{i-1}$ and $\lim _{k \to +\infty} 
g^k(z) = \om _{i}$.
\end{enumerate}

The lines $(\Delta_j)_{j \in \Z/n\Z\backslash \{i,i-1\} } $
and the straight (oriented) line $\Delta _{*}$ from $\alpha _{i-1}$ to $\om _{i}$ bound an $n-1$-
gon $P'$ such that $i(P') = i(P)$, and $g$ realizes $P'$.

\end{proof}

By applying the previous lemma inductively,  there exists  $g\in \homeo ^+ (\D)$ such that $\fix (g) = \fix (f)$ and
$g$ realizes a minimal $n$-gon $P$ such that $i(P)<0$, and 
$\delta _i =1$ for all $i\in \Z/n\Z$.

This next lemma finishes the proof of Theorem \ref{main}:

\begin{lema}  If $f$ realizes a minimal $n$-gon $P$ such that $i(P)<0$, and $\delta _i =1$ for all $i\in \Z/n\Z$,
then $\fix(f) \neq \emptyset.$
 
\end{lema}

\begin{obs}\label{orderh} With these assumptions,  the cyclic order of the points  $\{\alpha _i\}, \{\om _i\},$ at the
circle at infinity satisfies:

$$ \alpha_i\to \alpha_{i-1}\to \om_{i+1}\to \om_i\to 
\alpha_{i+2}$$ \noindent for all even values of $i\in \Z/2m\Z.$

\end{obs}

We apply  Lemma \ref{ends} and obtain a family of closed disks $(b_i'^l)_{l\in \Z\backslash \{0\}, i\in \Z/2m\Z}$. We are allowed to suppose that all the bricks $(b_i'^l)_{l\in \Z\backslash \{0\}, i\in \Z/2m\Z}$ have pairwise disjoint interiors 
(see Remark \ref{ptsdf}).
We construct a maximal free brick decomposition $(V,E,B)$ such that for all $i\in \Z/2m\Z$ and for all $l\in \Z\backslash \{0\}$, there 
exists $b_i^l\in B$ such that $b_i'^l\subset b_i^l$ (see Corollary \ref{bon}).\\

We will suppose that $f$ is not recurrent, and we will show that we can construct a hyperbolic configuration.

\begin{lema}\label{h1}{\bf (Hyperbolic domino effect)}  There exists $k>0$ such that for all even values of $i\in \Z/2m\Z$, 
both attractors
$[b_i^{-k}]_>$ and $[b_{i-1}^{-k}]_>$ contain $b^k_l$ for all $l\in \{i-2, i-1, i, i+1\}$.
 
\end{lema}

\begin{obs} Note that for all $i=0\mod 2$:
$$\om_{i-1}\to \om_{i-2}\to \alpha_i \to \alpha_{i-1} \to \om_{i+1} \to \om_i.$$ So,
the ``future indices'' $\{i-2, i-1, i, i+1\}$ are those coming immediately before and immediately after 
the ``past indices'' $\{i,i-1\}$ in the cyclic order.
 
\end{obs}

\begin{proof} By Remark \ref{puntas}, we can find an arc $$\Gamma : [0,1] \to \cup _{l\geq 1}[b_i^{-l}]_> \cup 
\{\alpha _i, \om_i\}$$\noindent joining $\alpha _i$ and $\om _i$. So, $\Gamma$ separates $\alpha _{i-1}$ from $\om _{i-1}$
and $\alpha _{i+1}$ from $\om _{i+1}$ (in $\overline \D$).  So, there exists $l>0$ such that $[b_i^{-l}]_> \cap [b_{i-1}^{l}]_<\neq \emptyset $ 
and $[b_i^{-l}]_> \cap [b_{i+1}^{l}]_<\neq \emptyset $. So, 
$$(\cup _{k\geq l} b_{i-1}^k) \cap (\cup _{k\geq l} b_{i+1}^k) 
\subset[b_i^{-l}]_>.$$\noindent  Using Remark \ref{puntas} again, we can find an arc 
$$\Gamma' : [0,1] \to [b_i^{-l}]_> \cup \{\om _{i+1}, \om_{i-1}\}$$\noindent joining $\om _{i+1}$ and $\om _{i-1}$. The cyclic
 order at $S^1$ of the points $\{\alpha _i\}, \{\om_i\}$,
 implies that $\Gamma '$ separates $\om _{i-2}$ from $\alpha _{i-2}$ in $\overline \D$.  So, 
$$\Gamma ' \cap \cup _{k\geq 1} [b_{i-2}^{k}]_< \neq \emptyset, $$\noindent  which implies that there exists $j>0$ such that
 $b_{i-2}^j\in [b_i^{-l}]_>$.
By taking $m>0$ large enough, we obtain that for all $l\in \{i-2, i-1, i, i+1\}$, $b^m_l \in [b_i^{-m}]_>$. Analogously we
obtain $b^p_l \in [b_{i-1}^{-p}]_>$  for all $l\in \{i-2, i-1, i, i+1\}$, for a suitable $p>0$. We finish by taking $k\geq \max 
\{m,p\}$
\end{proof}

We are now ready to prove Proposition \ref{simpleh}:

\begin{proof} We will show that $(([b_i^{-k}]_<)_{i=0\mod 2}, ([b_i^{k}]_>)_{i=0\mod 2})$ is a hyperbolic 
configuration, where $k>0$ is given by  Lemma \ref{h1} (the choice of even indices is arbitrary; we may as well have
 chosen the odd indices).  

By Remark \ref{orderh} and Lemma \ref{h1}, we just have to show that the sets 
$[b_i^{-k}]_<$, 
$[b_i^{k}]_>$, for $i$ even, are pairwise disjoint.   Lemma \ref{h1} also gives us, 
$$[b_i^{-k}]_<\cap [b_{i-2}^{k}]_>=\emptyset,$$ \noindent for $i$ even.
If
$[b_i^{-k}]_<\cap[b_j^k]_>\neq\emptyset$ for an even $j$ other than $i-2$, then we
can find an arc
$\Gamma : [0,1]\to [b_i^{-k}]_<\cup \{\alpha _{i}, \alpha _{j}\}$ joining $\alpha _{i}$ and 
$\alpha _j$.  The cylic order at $S^1$ of the points $\{\alpha_i\}, \{\om _i\}$ implies that $\Gamma$ separates
$\om _{i}$ from $\om _{i-2}$ in $\overline \D$.  As $[b_i^{-k}]_>$ is
a connected set whose closure contains both $\om _{i}$ and $\om _{i-2}$ (by the previous lemma), one gets 
$$[b_i^{-k}]_> \cap \Gamma \neq \emptyset$$ \noindent and so
$$[b_i^{-k}]_> \cap  [b_i^{-k}]_<\neq \emptyset , $$
\noindent which implies that $f$ is recurrent.  So, we have: 
$$[b_i^{-k}]_<\cap[b_j^k]_> = \emptyset ,  $$ \noindent for any pair of even indices $i,j$. 
We will show that $$[b_i^{-k}]_<\cap [b_j^{-k}]_<= \emptyset$$ \noindent for any two different even indices  $i,j$.  Otherwise,
 we could find an arc $$\Gamma : [0,1]\to [b_i^{-k}]_<\cup [b_j^{-k}]_< \cup
\{\alpha _{i}, \alpha _{j}\}$$ \noindent joining $\alpha _{i}$ and $\alpha _{j}$, from which we deduce again using the preceding
lemma that $$([b_i^{-k}]_<\cup [b_j^{-k}]_< )\cap [b_i^{-k}]_> \neq \emptyset .$$ 

\noindent So, as $f$ is not recurrent, we have
$$ [b_j^{-k}]_< \cap [b_i^{-k}]_> \neq \emptyset.$$  But now we can find  an arc $\Gamma : [0,1]\to [b_j^{-k}]_<
\cup \{\alpha _{i}, \alpha _{j}\}$ joining $\alpha _{i}$ and $\alpha _{j}$, which implies
$$[b_j^{-k}]_<\cap [b_j^{-k}]_>\neq \emptyset,$$ \noindent contradicting that $f$ is not recurrent.
The proof of the fact that $[b_i^{k}]_>\cap[b_j^k]_>= \emptyset$ for any two different even indices $i,j$, is completely
 analogous.
 
\end{proof}

\bibliography{jules}

\begin{thebibliography}{10}

\bibitem{brouwer}
L.~E.~J. Brouwer.
\newblock Beweis des ebenen translationssatzes.
\newblock {\em Math. Ann.}, 72:37--54, 1912.

\bibitem{brown}
M~Brown.
\newblock A new proof of $\textrm{B}$rouwer's lemma on translation arcs.
\newblock {\em Houston J. Math.}, 10:35--41, 1984.

\bibitem{brownkister}
M~Brown and J~Kister.
\newblock Invariance of complementary domains of a fixed point set.
\newblock {\em Proc. of the Am. Math. Soc.}, 91:503--504, 1984.

\bibitem{fathi}
A~Fathi.
\newblock An orbit closing proof of $\textrm{B}$rouwer's lemma on translation
  arcs.
\newblock {\em Enseign. Math.}, 2(33):315--322, 1987.

\bibitem{pbf}
J.~Franks.
\newblock Generalizations of the $\textrm{P}$oincar{\'e}-$\textrm{B}$irkhoff
  theorem.
\newblock {\em Ann. of Math.}, 128:139--151, 1998.

\bibitem{guillou}
L~Guillou.
\newblock Th{\'e}or{\`e}me de translation plane de $\textrm{B}$rouwer et
  g{\'e}n{\'e}ralisations du th{\'e}or{\`e}me de
  $\textrm{P}$oincar{\'e}-$\textrm{B}$irkhoff.
\newblock {\em Topology}, 33:331--351, 1994.

\bibitem{handel}
M.~Handel.
\newblock A fixed-point theorem for planar homeomorphisms.
\newblock {\em Topology}, 38:235--264, 1999.

\bibitem{duke}
P.~{Le Calvez}.
\newblock Periodic orbits of hamiltonian homeomorphisms of surfaces.
\newblock {\em Duke Math. J.}, 133(1):125--184, 2006.

\bibitem{patrice}
P.~{Le Calvez}.
\newblock Une nouvelle preuve du th{\'e}or{\`e}me de point fixe de
  $\textrm{H}$andel.
\newblock {\em Geometry \& Topology}, 10:2299--2349, 2006.

\bibitem{leroux}
F.~{Le Roux}.
\newblock {\em Hom{\'e}omorphismes de surfaces: th{\'e}or{\`e}mes de la fleur
  de $\textrm{L}$eau-$\textrm{F}$atou et de la variet{\'e} stable}.
\newblock Ast{\'e}risque, 2004.

\bibitem{sauzet}
A.~Sauzet.
\newblock Application des d{\'e}compositions libres {\`a} l'{\'e}tude des
  hom{\'e}omorphismes de surface.
\newblock {\em Th{\`e}se de l'Universit{\'e} Paris 13}, 2001.

\end{thebibliography}
\bibliographystyle{plain}

\author{$\ $ \\
Juliana Xavier\\
  I.M.E.R.L,\\
  Facultad de Ingenier\'ia,\\
Universidad de la Rep\'ublica,\\
  Julio Herrera y Reissig,\\
  Montevideo, Uruguay.\\
  \texttt{jxavier@fing.edu.uy}}

\end{document}